\documentclass[12pt,a4paper]{article}
\usepackage[margin=1in]{geometry} 
\usepackage{amsmath,amssymb,amsthm}
\usepackage{color}
\usepackage{graphicx}
\usepackage{subfigure}
\usepackage{booktabs}
\usepackage{threeparttable}
\usepackage{multirow}
\usepackage{array}
\usepackage{amsmath,amsthm,amssymb,amscd}
\usepackage{latexsym}
\usepackage{indentfirst}
\usepackage{bibentry}
\usepackage{textcomp}
\usepackage{float}
\usepackage{multirow}
\usepackage{booktabs}
\usepackage{graphicx}
\usepackage{color}
\usepackage{cite}
\usepackage{authblk}
\usepackage{algorithmicx,algorithm}
\usepackage{listings}
\usepackage{xcolor}
\usepackage{alltt}
\usepackage{enumitem}
\usepackage{lscape}
\usepackage{geometry}
\usepackage{enumitem}
\usepackage{longtable}
\usepackage{afterpage}
\usepackage{capt-of}
\newtheorem{theorem}{Theorem}[section]
\newtheorem{lemma}[theorem]{Lemma}
\newtheorem{proposition}[theorem]{Proposition}
\newtheorem{corollary}[theorem]{Corollary}

\newtheorem*{rem}{Remark}

\numberwithin{equation}{section}

\newcommand{\be}{\begin{equation}}
\newcommand{\ee}{\end{equation}}

\newcommand{\bea}{\begin{eqnarray}}
\newcommand{\eea}{\end{eqnarray}}
\newcommand{\bean}{\begin{eqnarray*}}
\newcommand{\eean}{\end{eqnarray*}}

\newcommand{\bg}{\begin{gather}}
\newcommand{\eg}{\end{gather}}
\newcommand{\bgn}{\begin{gather*}}
\newcommand{\egn}{\end{gather*}}

\title{Orthogonal Polynomials, Asymptotics  and Heun Equations}
\author{Yang Chen$^a$\thanks{yayangchen@umac.mo},~
Galina Filipuk$^{b,}$\thanks{filipuk@mimuw.edu.pl},~Longjun Zhan$^{a,}$\thanks{Corresponding author: Zhan\_Longjun@foxmail.com}\\
\footnotesize{$^{a}$Department of Mathematics, University of Macau,}\\
\footnotesize{Avenida da Universidade, Taipa, Macau, China}\\
\footnotesize{$^{b}$Faculty of Mathematics, Informatics and Mechanics,}\\
\footnotesize{University of Warsaw, Banacha 2, Warsaw, 02-097, Poland}}

\date{}							

\begin{document}
\maketitle

\begin{abstract}
The Painlev\'{e} equations  arise from the study of Hankel determinants generated by moment matrices, whose
weights are expressed as the product of ``classical" weights multiplied by suitable ``deformation factors", usually dependent on a ``time variable'' $t$.
From ladder operators \cite{ChenIts2010,ChenIsmail2005,ChenIsmail1997,Magnus1995} one finds second order linear ordinary differential equations  for associated orthogonal polynomials with coefficients being rational functions.  The Painlev\'e and related functions appear as the residues of these rational functions.

We will be interested in the situation when $n$, the order of the Hankel matrix and also the degree of the  polynomials $P_n(x)$ orthogonal with respect to the deformed weights, gets large. We show  that the second order linear differential equations satisfied by $P_n(x)$ are particular cases of Heun equations when $n$ is large. In some sense,  monic orthogonal polynomials generated by deformed weights mentioned below are solutions of a variety of Heun equa\-tions. Heun equations are of considerable importance in mathematical physics and in the special cases they degenerate to the hypergeometric and confluent hypergeometric equations (see, for instance, \cite{Handbook,EL,HDE}).

In this paper we look at three type of weights:  the Jacobi type,
the Laguerre type 
and the weights deformed by the indicator function of $(a,b)$ $\chi_{(a,b)}$    and the step function $\theta(x)$.

\vskip .2cm
\noindent
In particular, we consider the following Jacobi type weights:
\\
$1.1)\: x^\alpha(1-x)^\beta{\rm e}^{-t x},\;x\in[0,1],\;\alpha,\beta,t>0$;
\\
$1.2)\:x^\alpha(1-x)^\beta{\rm e}^{-t/ x},\; x\in(0,1],\; \alpha,\beta,t>0$;
\\
$1.3)\:(1-x^2)^\alpha(1-k^2x^2)^\beta,\;x\in[-1,1], \;\alpha,\beta>0,\;k^2\in(0,1);$
\vskip .2cm
\noindent
the Laguerre type weights:
\\
$2.1)\:x^\alpha(x+t)^\lambda{\rm e}^{-x},\; x\in[0,\infty),\;t,\alpha,\lambda  >0;$
\\
$2.2)\:x^\alpha {\rm e}^{-x-t/x},\;x\in(0,\infty),\alpha,t>0.$
\vskip .3cm

We also study another type of deformation when the classical weights are multiplied by $\chi$ or $\theta$:

\noindent
$3.1)\:{\rm e}^{-x^2}(1-\chi_{(-a,a)}(x)),\;x\in\mathbb{R},\;a>0;$

\noindent
$3.2)\:(1-x^2)^\alpha(1-\chi_{(-a,a)}(x)),\;x\in[-1,1],\; a\in(0,1),\;\alpha>0;$

\noindent
$3.3)\:x^\alpha{\rm e}^{-x}(A+B\theta(x-t)),\:\: x\in[0,\infty),\;\alpha,\,t>0,\;A\geq0, \;A+B\geq0.$

 The weights mentioned above were studied in a series of papers related to the deformation of ``classical" weights   \cite{EC,BasorChen,CCF,CD,ChenIts2010,LC2,LC1,MC2018,ZGCZ}.



\end{abstract}

\section{Introduction}

\subsection{Heun equations}

The general Heun equation is the second order linear Fuchsian ordinary differential equation with four regular singular points in the complex plane \cite{Heun1889,HDE,S2000,Slee}. It is a generalization of the well-studied Gauss hypergeometric equation with three regular singularities. However, it is much harder to study properties of the Heun functions. The additional singularity causes many complications in comparison with the hypergeometric case (for instance, solutions do not have integral representation). There also exist confluent Heun equations, see \cite{HDE,S2000}, which are obtained by certain confluence of singularities of the    general Heun equation.

The general Heun equation is given by
\begin{equation}\label{Heun}
  \quad\frac{d^2y}{d z^2}+\left(\frac{\gamma}{z}+\frac{\delta}{z-1}+\frac{\epsilon}{z-a}\right)\frac{dy}{d z}+\frac{\alpha\beta z-q}{z(z-1)(z-a)}y=0,
\end{equation}
where the parameters satisfy the Fuchsian relation
\begin{equation}\label{Fuchr}
1+\alpha+\beta=\gamma+\delta+\epsilon.
\end{equation}
This equation has four regular singular points at $z=0,1,a$ and $\infty$. Its solutions, the Heun functions, are usually denoted by $y=H(a,q;\alpha,\beta,\gamma,\delta;z)$, where $\epsilon$ is expressed in term of $\alpha, \beta, \gamma, \delta$ via \eqref{Fuchr}. The parameter $q$ is called an accessory parameter.

It is well-known that the derivative of the hypergeometric function $_2F_1$ is again a hypergeometric function with different values of the parameters. However, for the Heun function it is generally not the case. The first order derivative of the general Heun function satisfies the second order Fuchsian differential equations with five regular singular points. It can be verified by direct computations that the function $v(z)=dy/dz$, where $y=y(z)$ is a solution of \eqref{Heun}, satisfies the following equation:
\begin{equation}\label{dHeun}
\frac{d^2v}{d z^2}+\left(\frac{\gamma+1}{z}+\frac{\delta+1}{z-1}+\frac{\epsilon+1}{z-a}-\frac{\alpha\beta}{\alpha\beta z-q}\right)\frac{dv}{d z}+\frac{f(z)}{z(z-1)(z-a)(\alpha\beta z-q)}v=0,
\end{equation}
where $f(z)=z(\alpha\beta z-2q)(\alpha\beta +\gamma+\delta+\epsilon)+q^2+q(\gamma+a(\gamma+\delta)+\epsilon)-\alpha\beta\gamma a$. We see from the equation above that an additional singularity at $z=q/(\alpha\beta)$ appear.

 There are four confluent limits of the general Heun equation: the confluent Heun, double confluent Heun, bi-confluent Heun and tri-confluent Heun equations.
 When the singularity $z=a$ is merged with $z=\infty$ the confluent Heun equation is found. Translating $z=1$ to $z=b$ then followed by $a\rightarrow\infty,\;\;b\rightarrow 0$ one finds the double confluent Heun equation. The bi-confluent Heun equation is obtained by $a\rightarrow\infty,\;\;b\rightarrow \infty.$ The tri-confluent Heun equation cannot be derived directly by confluence from the Heun equation in its standard form, we should be go back to the less specialized parametrization with singularities at $a_1,\;a_2,\;a_3$ and $\infty$, which is followed by $a_j\rightarrow\infty,\; j=1,2,3$.  These transformations are due to Heun (1889) (see\cite{HDE}) and they can be checked in  Maple\footnote{https://www.maplesoft.com/}. Other transformations can also be found in the literature, see for example, Slavyanov and Lay \cite{S2000}.

The list of Heun equations and equations for derivatives of the Heun functions is as follows.

{\bf
The confluent Heun equation} is given by
\begin{equation}\label{}
\frac{d^2y}{d z^2}+\left(\frac{\gamma}{z}+\frac{\delta}{z-1}+\epsilon\right)\frac{dy}{d z}+\frac{\alpha z-q}{z(z-1)}y = 0
\end{equation}
and the linear equation for the function $v=dy/dz$ is given by
\begin{equation}\label{dC}
\frac{d^2v}{d z^2}+\left(\frac{\gamma+1}{z}+\frac{\delta+1}{z-1}+\epsilon-\frac{\alpha}{\alpha z-q}\right)\frac{dv}{d z}+\frac{g(z)}{z(z-1)(\alpha z-q)}v = 0,
\end{equation}
where $g(z)=(\alpha+\epsilon)(\alpha z-2q )z+q^2-(\gamma+\delta-\epsilon)q+\alpha\gamma$.

{\bf The double-confluent Heun equation} is given by
\begin{equation}\label{}
 \frac{d^2y}{d z^2}+\left(\frac{\gamma}{z^2}+\frac{\delta}{z}+\epsilon\right)\frac{dy}{d z}+\frac{\alpha z -q}{z^2}y=0
\end{equation}
and the linear equation for the function $v=dy/dz$ is given by
\begin{equation}\label{dD}
\frac{d^2v}{d z^2}+\left(\frac{\gamma}{z}+\frac{\delta+2}{z-1}+\epsilon-\frac{\alpha}{\alpha z-q}\right)\frac{dv}{d z}+\frac{h(z)}{z^2(\alpha z-q)}v = 0,
\end{equation}
where $h(z)=(\alpha+\epsilon)(\alpha z-2q)z+q^2-\gamma q-\alpha\gamma$.

{\bf The bi-confluent Heun equation} is given by
\begin{equation}\label{}
 \frac{d^2y}{d z^2}+\left(\frac{\gamma}{z}+\delta+\epsilon z\right)\frac{dy}{d z}+\frac{\alpha z -q}{z}y=0
\end{equation}
and the linear equation for the function $v=dy/dz$ is given by
\begin{equation}\label{}
\frac{d^2v}{d z^2}+\left(\frac{\gamma+1}{z}+\delta+\epsilon z-\frac{\alpha}{\alpha z-q}\right)\frac{dv}{d z}+\frac{k(z)}{z(\alpha z-q)}v = 0,
\end{equation}
where $k(z)=(\alpha+\epsilon)(\alpha z-2q)z+q^2-\gamma q-\alpha\gamma$.

{\bf The tri-confluent Heun equation} is given by
\begin{equation}\label{}
 \frac{d^2y}{d z^2}+\left(\gamma+\delta z+\epsilon z^2\right)\frac{dy}{d z}+(\alpha z -q)y=0
\end{equation}
and the linear equation for the function $v=dy/dz$ is given by
\begin{equation}\label{}
\frac{d^2v}{d z^2}+\left(\gamma+\delta z+\epsilon z^2-\frac{\alpha}{\alpha z-q}\right)\frac{dv}{d z}+\frac{p(z)}{\alpha z-q}v = 0,
\end{equation}
where $p(z)=(\alpha+\epsilon)(\alpha z-2q)z+q^2-\gamma q-\alpha\gamma$.

Solutions of the Heun equation were given by Heun in 1889. In the 50 years following 1889 \cite{Heun1889},  solutions of Heun equations were obtained as power series in $z$.
Svartholm (1939) \cite{S1939} showed that solutions of the Heun equations may also be represented as series of hypergeometric functions.  This was further developed by Erd\'{e}lyi (1942--1944) \cite{E1942a}--\cite{E1944}.  Schmidt (1979) \cite{S1979} included the possibility of doubly infinite series of hypergeometric functions, similar to Laurent series. Kalnins and Miller (1991)  \cite{K1991} were concerned with the expansion of Heun polynomials based on group-theoretic methods and technique of separation of variables on the $n$-sphere and deduced the expansion of a product of two Heun polynomials  in terms of the product  of Jacobi polynomials. There are certain cases where the confluent Heun functions could  be expressed in terms of special functions of mathematical physics (see, for instance recent results by A. Ishkhanyan, e.g., \cite{Ish}).
  Hence, it is of considerable interest to construct solutions of the Heun equations. In this paper we will describe the orthogonal polynomials with respect to the deformed weights and show that for large $n$ the ordinary differential equations that they satisfy 
 are Heun equations of various types.

\subsection{Orthogonal polynomial and ladder operators}

Let  $\{P_j(x)\}_{j=1}^{\infty}$  be a sequence of monic polynomials of degree $j$ orthogonal with respect to the weight $w(x)$ on the interval $[a,\,b]$, i.e.,
$$\int_a^bP_j(x)P_k(x)w(x){\rm d}x=h_j\delta_{j,k},\quad j,k=0,1,2,\cdots,$$
where $h_j$ denote the square of the weighted    $L^2$ norm of $P_j(x)$ over $[a,b]$. We write $$P_n(x)=x^n+p(n)x^{n-1}+\ldots.$$
It is known that
$$\prod _{1\leq i<j\leq N}(x_i-x_j)=\det\left(x_j^{i-1}\right)_{i,j=1}^{N}=\det\left(P_{i-1}(x_j)\right)_{i,j=1}^N .$$ The polynomials $P_n(x)$ can be constructed by the Gram-Schmidt orthogonalization process. Referring to the weights listed in the abstract, we see that the parameters $t,\, \alpha,\ldots$ will also appear in the polynomials and in the $L^2$ norm. However, to simplify notations we will not usually display all the dependence.

If  the moments of the deformed weight $\mu_j=\int_a^b x^j w(x) {\rm d} x$ exist, and $$\det (\mu_{i+j})_{0\leq i,j\leq n-1}\neq0,$$
then the theory of orthogonal polynomials states that monic orthogonal polynomials $P_n(z)$ for $n=0,1,2,\ldots$ satisfy the three-term recurrence relation
$$zP_n(z)=P_{n+1}(z)+\alpha_nP_n(z)+\beta_nP_{n-1}(z)$$ with $$P_0(z)=1,\quad\beta_0 P_{-1}(z)=0.$$
The  recurrence coefficients $\alpha_n$ and $\beta_n$ will depend on the parameters of the weight $t,\,\alpha $ etc.
 Note that the monic  polynomials $P_j(x)$ orthogonal with respect to weight $w(x)$ are defined on the real axis. However, they can be extended to the complex plane, hence,  we will use the variable  $z$  in $P_j(z)$.
For more detail about orthogonal polynomials see Szeg\"{o} \cite{Szego}.

Let $w(a)=w(b)=0$.  In \cite{ChenIsmail1997} it is shown that the following relations hold.

\begin{lemma}
Assume that  $v(x)=-\log w(x)$ has derivative in some Lipschitz class with a positive exponent. The lowering and raising operators (ladder operators) satisfy the following differential-difference formulas:
\begin{align}
  P'_n(z) &=-B_n(z)  P_n(z)+\beta_n A_n(z)P_{n-1}(z), \label{UP}\\
  P'_{n-1}(z) &=\left[B_n(z)+v'(z)\right]P_{n-1}(z)-A_{n-1}(z)P_n(z),\label{DOWN}
\end{align}
where
\begin{align}
&A_n(z):=\frac{1}{h_n}\int_a^b\frac{v^\prime(z)-v^\prime(y)}{z-y}P^2_n(y)w(y){\rm d}y,\label{A}\\
&B_n(z):=\frac{1}{h_{n-1}}\int_a^b\frac{v^\prime(z)-v^\prime(y)}{z-y}P_n(y)P_{n-1}(y)w(y){\rm d}y.\label{B}
\end{align}
\end{lemma}

If $w(a)\neq0,\, w(b)\neq0$, then additional terms should be included in the definitions of $A_n(z)$ and $B_n(z)$  (See Chen and Ismail \cite{CI1997}, \cite{ChenIsmail2005}).
The variable $z$ shown in the    equations above   is complex; we assume that $v'(z)$ is an extension of $v'(x)$  off the real axis.

\begin{lemma}
The functions $A_n(z)$ and $B_n(z)$ defined by \eqref{A} and \eqref{B} satisfy the identities
\begin{align}
&B_{n+1}(z)+B_n(z)=(z-\alpha_n)A_n(z)-  v'(z),~~~~~~~~~~~~~~~~~~~~~~~~~~~~~~~~~~~~~(S_1)\notag\\
&1+  (z-\alpha_n)[B_{n+1}(z)-B_n(z)]=\beta_{n+1}A_{n+1}(z)-\beta_nA_{n-1}(z).~~\;~~~~~~~~~~~~~(S_2)\notag
\end{align}
\end{lemma}

It turns out that there is another supplementary condition involving $\sum _{j=0}^{n-1} A_j(z)$, we will call it  $(S_2')$, which is widely used in the determination of recurrence coefficients $\alpha_n$ and $\beta_n$,
\begin{equation}\label{}
B_n(z)^2+v'(z)B_n(z)+\sum_{j=0}^{n-1}A_j(z)=\beta_nA_n(z)A_{n-1}(z).~~~~~~~~~~~~~~~~~~~~~~~~~~(S_2')\notag
\end{equation}
Equation $(S_2')$ should be thought of as an equation for  $\sum_{j=0}^{n-1}A_j(z)$. See, for example, Basor, Chen \cite{BC2009} and Chen, Its \cite{ChenIts2010}.

Eliminating $P_{n-1}(z)$ from ladder operators and we obtain the second order linear ordinary differential equation satisfied by $P_n(z)$
\begin{equation}\label{H}
P_n''(z)-\left( v'(z)+\frac{A'_n(z)}{A_n(z)}\right) P_n'(z)+\left(
B'_n(z)-B_n(z)\frac{A'_n(z)}{A_n(z)}+\sum_{j=0}^{n-1}A_j(z)\right)P_n(z)=0,
\end{equation}
 where $\sum_{j=0}^{n-1}A_j(z)$ is obtained from $(S'_2)$.

In the following sections we will discuss the second order linear ordinary differential equations for large $n$, in the context of the deformed weights.

\subsection{Coulomb fluid method}

In this section we treat the joint distribution function of the eigenvalues of the Hermitian ensembles as points of a fluid described by a continuous density $\rho(x) {\rm d}x$. We first present some basic description of the Coulomb fluid method, mainly from \cite{CL1998,CI1997}.
The quanlity

\begin{equation*}
  E(x_1,x_2,\cdots,x_n)=-2\sum_{1\leq j<k\leq n}{\rm ln}\mid x_j-x_k\mid+\sum_{j=1}^n v(x_j),
\end{equation*}
is
%
the total energy of a system of $n$  logarithmical repelling particles in one dimension subject to an external potential $v(x)$. The particles can be approximated as a continuous fluid with a density $\rho$, for sufficiently large $n$. This density $\rho(x)$ assumed to be supported on $[a,b]$ will correspond to the equilibrium density of the fluid, this is obtained by the constrained minimization
$$\min_{ \rho>0} F[\rho] \quad \text{subject to}\quad \int_a^b \rho(x){\rm d}x=n,$$
where the free-energy function reads
\begin{equation*}
  F[\rho]:=\int_a^b\rho(x)v(x){\rm d}x-\int_a^b\int_a^b\rho(x) {\rm ln} |x-y|\rho(y){\rm d}x{\rm d}y.
\end{equation*}
The equilibrium density satisfies the following integral equation (see the Frostman Lemma \cite{T1959}):
$$v(x)-2\int_a^b{\rm ln}|x-y|\rho(y){\rm d}y=A,\qquad x\in[a,b],$$
where $A$ is the Lagrange multiplier   that fixes the constraint $\int_a^b \rho(x){\rm d}x=n$. For more detail see \cite{CI1997}. After taking a derivative with respect to $x$ one obtains a singular integral equation,
\begin{equation}\label{Singular1}
2 P\int_a^b\frac{\rho(y)}{x-y}{\rm d}y=v'(x),
\end{equation}
where $P$ denotes the Cauchy principal value. According to the standard theory of singular integral equations \cite{FD1990,NI2013}, if $\rho(a)=\rho(b)=0$, then the density supported on $[a,b]$ reads
\begin{equation}\label{sigma}
\rho(x)=\frac{\sqrt{(b-x)(x-a)}}{2\pi^2}P\int_a^b {\rm d}y\frac{v'(y)}{(y-x)\sqrt{(b-y)(y-a)}}.
\end{equation}

The endpoints of the interval $[a,b]$ satisfy the condition $\int_a^b \rho(x){\rm d}x=n$, as well as  stability conditions
 \begin{equation}\label{supp2}
\begin{split}
   \int_a^b  \frac{xv'(x)}{\sqrt{(b-x)(x-a)}}{\rm d}x&=2n\pi,\\
  \int_a^b  \frac{v'(x)}{\sqrt{(b-x)(x-a)}}{\rm d}x&=0 .
  \end{split}
\end{equation}
The end points of the support of the density are the solutions of \eqref{supp2} and are denoted by $a(n,t)$ and $b(n,t)$. They  depend on the independent variables $n$, $t$, which play an important role in the asymptotics of the recurrence coefficients $\alpha_n(t)$ and $\beta_n(t)$, with
  {blue}\begin{align}
\alpha_n(t)&\sim\frac{a(n,t)+b(n,t)}{2},\label{asym1}\\
\beta_n(t)&\sim\left(\frac{b(n,t)-a(n,t)}{4}\right)^2,\label{asym2}
\end{align}
see \cite{CI1997, MNT1987}.

\subsection{The structure of this paper}

The second order linear differential equations satisfied by $P_n(z)$  related to several weight functions in the abstract, \eqref{H}, have coefficients that are rational functions of $z$, whose poles and residues depend on $r_n(t)$ and $R_n(t)$. Here $t$ is a ``time parameter". It was found that $r_n(t)$ and $R_n(t)$ are evaluated as the ``matrix elements" involving $P_n^2(x,t)$ and $P_n(x,t)P_{n-1}(x,t)$.
 In the ladder operators \eqref{UP} and \eqref{DOWN}, with the weights given in the abstract, the functions $A_n(z)$ and $B_n(z)$ are rational function of $z$. Conditions $(S_1)$, $(S_2)$ and $(S'_2)$ are used to obtain  relations for recurrence coefficients $\alpha_n$, $\beta_n$ and auxiliary quantities $R_n$ and $r_n$.
 In particular, 
 one finds that the recurrence coefficients $\alpha_n(t)$ and $\beta_n(t)$ are expressed in terms of the auxiliary variables $r_n(t)$ and $R_n(t)$, which typically satisfy the coupled Riccati equations. Eliminating $r_n(t)$ gives a nonlinear second order ordinary differential equation for the function $R_n(t)$, which turns out to be equivalent   (possibly after some change of variables or scaling) to one of  the classical Painlev\'{e} equations.  

We show that in the situation where $n$ tends to $\infty$, the linear second order ordinary differential equations \eqref{H} turn out to be Heun equations.
 The large $n$ behaviour of $R_n(t)$ is found by using the non-derivative part of the equations  satisfied by $R_n(t)$. From this approximation, we obtain the behaviour of $r_n(t)$ and $R_n(t)$ under suitable double scaling and finally compute the recurrence coefficients, $\alpha_n$, $\beta_n$. We see that the behaviour of the recurrence coefficients obtained by this method is accurate and compare very well with the behaviour of recurrence coefficients obtained from   \eqref{supp2}, \eqref{asym1} and \eqref{asym2}.


This paper is organized as follows. In Sections 2--4 we study the deformed Jacobi type weights, deformed Laguerre type weights and weights with gaps respectively. We write the  second order linear ordinary differential equations satisfied by orthogonal polynomials $P_n(z)$, which are usually known from the corresponding literature.  Then we deduce the Heun equations via some approximation  procedure.  
The main results of the paper are summarized in the following table.

{\clearpage
 \thispagestyle{empty}
     \begin{landscape}
     \captionof{table}{\footnotesize Equations with respect to weight $w(x)$, when $n$ goes to infinity.}
    \label{}
     \setlength{\tabcolsep}{0.02pt}

       \renewcommand{\arraystretch}{1.8}

    {\scriptsize \begin{longtable}[l]{c|c}
    \toprule
     WEIGHT&EQUATION\cr
    \midrule
    \multirow{2}{*}
    { Sec 2.1: $ x^\alpha(1-x)^\beta{\rm e}^{-t x},\;x\in[0,1],\;\alpha,\beta,t>0.$}

       &{  $n\rightarrow\infty$, $t\rightarrow0^+$, $T=n t$, confluent Heun equation.
}\cr

\cline{2-2}
                  &{$
  P_n''(u)+\left(\frac{\alpha+1}{u}+\frac{\beta+1}{u-1}-\frac{T}{n}\right)P_n'(u)+\frac{T u-T/2-n(n+\alpha+\beta+1)}{u(u-1)}P_n(u)=0$}\cr
     \midrule
    \multirow{2}{*}
    {Sec 2.2: $x^\alpha(1-x)^\beta{\rm e}^{-t/ x},\, x\in(0,1],\;\alpha,\beta, t>0.$}     &{ $n\rightarrow\infty$, $t\rightarrow0^+$,  $T=2n^2 t$ small, confluent Heun equation.}\cr
\cline{2-2}
                  &$
  P_n''(u)+\left(\frac{2\lambda-\alpha-\beta }{u}+\frac{\beta+1}{u-1}-s\right)P_n'(u) +\frac{-s\lambda u+\lambda(\lambda+s-\alpha)+a}{u(u-1)}P_n(u)=0$\cr
    \midrule
    \multirow{3}{*}
    { Sec 2.3: $(1-x^2)^\alpha(1-k^2x^2)^\beta, x\in[-1,1],\; \alpha,\beta>0,\;k^2\in(0,1).$}     &{ $k^2\rightarrow0^+,\; \beta\rightarrow\infty,\;n\to\infty, \,k^2\beta=t \text{ is fixed}$, confluent Heun equation. }\cr
\cline{2-2}
                  &$
  P''_n(u)+\left(\frac{1}{2u}+\frac{\alpha +1}{u-1}-t\right)P'_n(u)+\frac{2ntu-n(n+2\alpha+1+t)}{4u(u-1)}P_n(u)=0.$\cr
    \midrule
     \multirow{2}{*}
    {Sec 3.1: $  x^\alpha(x+t)^\lambda{\rm e}^{-x}, x\in[0,\infty),\;t,\alpha,\lambda>0.$}     &{ $n\rightarrow\infty$, confluent Heun equation.}\cr
\cline{2-2}
                  &$
  P_n''(u)+\Big(\frac{\alpha+1}{u}+\frac{2\eta+\lambda}{u-1}+t\Big)P_n'(u) +\left(\frac{4t(\eta-n)u+(2 \eta +\lambda ) (2 (\alpha +\eta )+\lambda )-4 \lambda  \sqrt{nt}+4 n t+2 \lambda  t }{4 u (u-1)}\right)P_n(u)=0.$\cr
    \midrule
     \multirow{4}{*}
    {Sec 3.2: $x^\alpha {\rm e}^{-x-t/x}, x\in(0,\infty),\alpha,t>0.$}
    &{$n\rightarrow\infty, t\rightarrow 0^+$, $s=(2n+\alpha+1)t$ fixed, double confluent Heun equation.}\cr
\cline{2-2}
                  &For large $s$, $
  P_n''(u)+\Big(\frac{\gamma}{u^2}+\frac{\delta}{u}+\epsilon\Big)P_n'(u) +\frac{au-q}{ u^2}P_n(u)=0$.\cr
\cline{2-2}
                  &For small $s$, $
                  P_n''(u)+\Big(\frac{s(1+\alpha)}{2\alpha n u^2}+\frac{\alpha+1}{u}-1\Big)P_n'(u) +\frac{n u+s/2\alpha}  {u^2}P_n(u)=0$.\cr
    \midrule
    \multirow{2}{*}
    {Sec 4.1: ${\rm e}^{-x^2}(1-\chi_{(-a,a)}(x)),\;x\in\mathbb{R},\;a>0.$ }   &{ $n\rightarrow\infty$, confluent Heun equation.}\cr
\cline{2-2}
                  &$
  P_n''(u)+\Big(\frac{1} {u-1}-\frac{1}{2u}-t\Big)P_n'(u)+\frac{2ntu+\sqrt{2nt}}{4u(u-1)}P_n(u)=0.$\cr
   \midrule
    \multirow{2}{*}
    {Sec 4.2:  $(1-x^2)^\alpha(1-\chi_{(-a,a)}(x)), \;x\in[-1,1],\; a\in(0,1),\;\alpha>0.$}     &{ $n\rightarrow\infty$, general Heun equation.}\cr
\cline{2-2}
                  &$P_n''(u)+\left(-\frac{1}{2u}+\frac{\alpha +1}{u-1}+\frac{1}{u-t}\right)P_n'(u)-\frac{n \sqrt t+u (2 \alpha +n+1)n}{4u\left(u-1\right) \left(u-t\right)}P_n(u)=0.$\cr
    \midrule
    \multirow{4}{*}
   {Sec 4.3: $   x^\alpha{\rm e}^{-x}(A+B\theta(x-t)), x\in[0,\infty),\;\alpha, t>0,\; A\geq 0, \;A+B\geq0.$}     &{   For $A=0,\;B=1$, $n\rightarrow\infty$, $s=4n t$, for large $s$, obtain double confluent Heun equation.}\cr
\cline{2-2}
                  &$
  P_n''(u)+\Big(\frac{s-\alpha\sqrt s}{4nu^2}+\frac{\alpha+1}{u}-1\Big)P_n'(u)+\frac{4n u-2\alpha\sqrt s+s+\alpha^2}{4u^2}P_n(u)=0.$\cr
  \cline{2-2}
  &{   For $A=-1,\;B=1$, $n\rightarrow\infty$, confluent Heun equation.}\cr
\cline{2-2}
                  & $
   P_n''(u)+\Big(\frac{\alpha +1}{u}+\frac{1}{u-1}-t\Big)P_n'(u)+\frac{ntu- n( n+\alpha+1+t/2 )}{u (u-1) }P_n(u)=0.$\cr
    \bottomrule
        \end{longtable}}
    \end{landscape}
}

\section{Jacobi type weights}\label{SectionJ1}
In this section we consider three deformed Jacobi type weights: $x^\alpha(1-x)^\beta{\rm e}^{-t x}$ (see \cite{ZGCZ}), $x^\alpha(1-x)^\beta{\rm e}^{-t/ x}$ (see \cite{CCF,CD}) and $(1-x^2)^\alpha(1-k^2x^2)^\beta$ (see \cite{BasorChen}). The properties of  polynomials orthogonal  with respect to these weight and of their recurrence coefficientss were studied in corresponding papers. Moreover, it was shown there that the auxiliary quantities $R_n(t),\;r_n(t)$, closely  related to the  recurrence coefficients $\alpha_n(t),\;\beta_n(t)$, satisfy certain  Painlev\'{e} equations and Jimbo-Miwa-Okamoto $\sigma$-forms of the  Painlev\'{e} equations.
We show that  monic orthogonal polynomial $P_n(z)$ for the weights above satisfy  particular confluent Heun equations with parameters related to the parameters in the weight  as  $n$ goes to infinity. 
We will study these Jacobi type weights in the following three subsections.

\subsection{$x^\alpha(1-x)^\beta{\rm e}^{-t x},\; x\in[0,1],\;\alpha,\beta,t>0$ }

In \cite{ZGCZ} the probability density function of the center of mass $\mathbb{P}(c,\alpha,\beta,n)$ was studied. The second order linear differential equation satisfied by  monic polynomials orthogonal  with respect to $x^\alpha(1-x)^\beta{\rm e}^{-t x}$  is the Fuchsian equation with four singular points given by
\begin{align}\label{ZP1}
P_n''(z)&+Q_n(z,t)P_n'(z) +S_n(z,t)P_n(z)=0,
\end{align}
where
\begin{align*}
 Q_n(z,t) &=\frac{\alpha+1}{z}+\frac{\beta+1}{z-1}-t-\frac{1}{z-R_n(t)/t}, \\
  S_n(z,t)&=\frac{nt z-n-\sum_{j=0}^{n-1}R_j(t)}{z(z-1)}+\frac{n z+r_n(t)}{z(z-1)(z-R_n(t)/t)}
\end{align*}
with $R_n(t)$, $r_n(t)$ and $\beta_n(t)$ defined in \cite{ZGCZ}. They satisfy the following relations:
\begin{align}
r_n(t)=&\frac{1}{2t}\left[t R_n'(t)+\alpha t-(2n+1+\alpha+\beta+t)R_n(t)+R_n^2(t)\right] \label{Zr},\\
\sum_{j=0}^{n-1}R_j(t)=&n (\alpha +\beta +n)-t r_n(t)-t^2 \beta _n(t),\label{ZSR}
\end{align}
where
\begin{equation}
\beta_n(t)=\frac{n (\beta +n)+tr_n(t) \left[r_n(t)-\alpha  \right]/R_n(t)+(\alpha +\beta +2 n) r_n(t)}{t \left(t-R_n(t)\right)}.\label{Zb}
\end{equation}

In order to further study the asymptotic expression of the second order differential equation \eqref{ZP1}, we will first find the asymptotic expression of $R_n(t)$, see below Proposition~\ref{Zpropo2}.  For convenience of the reader we will use hollow symbols to define the new variable functions, such as $\mathbb{R}_n(T):=R_n(\frac{T}{n})$ and $\mathbb{X}(T):=X(\frac{T}{n})$.

\begin{proposition}\label{Zpropo2}
When $n\rightarrow\infty$, $t\rightarrow0^+$ and $T=t n$ is fixed,
\begin{align}
\mathbb{R}_n(T)=&2n+\alpha +\beta +1+\frac{ T}{2n}+\frac{T \left(\beta ^2-\alpha ^2\right)}{8n^3}+\frac{T (\alpha^2 -\beta^2 )  (\alpha +\beta +1)}{8 n^4}\notag \\
&+\frac{ \left(2 \alpha ^2+2 \beta ^2-1\right)T^2-(\alpha^2 -\beta^2 )\left(3 (\alpha+\beta) ^2+6( \alpha  + \beta) +4\right)T}{32 n^5}\notag\\&+\mathcal{O}\left(\frac{1}{n^4}\right). \label{ZR}
\end{align}

\begin{proof}
The auxiliary quantity
$$R_n(t):=\frac{\alpha}{h_n}\int_0^1\frac{P_n^2(x)}{x}x^\alpha(1-x)^\beta {\rm e}^{-t x}{\rm d}x$$
satisfies the second order nonlinear differential equation
\begin{align}\label{nonlinear eq Rn section 2.1}
R''_n=&\frac{1}{2t^2(R_n-t)R_n}\big\{(2R_n-t )(t R'_n)^2-2t R_n^2R'_n+2R_n^5-2\alpha^2t^2R_n+\alpha^2t^3 \notag \\
&-[2(2n+1+\alpha+\beta)+5t]R_n^4+4t(2n+1+\alpha+\beta+t)R_n^3\notag\\
&-[t^3+2 t^{2}(2n+1+\alpha+\beta)-t(1+\alpha^2-\beta^2)]R_n^2\big\},
\end{align}
 see \cite{BC2010,ZGCZ}, and it can further be reduced to the fifth Painlev\'e equation.

When $n$ (the dimension of the Hankel determinant) tends to infinity, $t$ tends to zero and the product of $n$ and $t$ is fixed $nt=T$, the function $\mathbb{R}_n(T)$ satisfies the following equation:
\begin{align}\label{tR}
n^3\frac{ \mathbb{R}_n{}^2}{T^2}&+n^2 \frac{\mathbb{R}_n\left(\alpha \mathbb{R}_n+\beta  \mathbb{R}_n-\mathbb{R}_n{}^2+\mathbb{R}_n-2 T \right)}{2 T^2}\notag \\
&-n \frac{\alpha ^2-\beta ^2+1-3 \mathbb{R}_n{}^2+2 (\alpha+\beta+1)  \mathbb{R}_n}{4 T }\notag \\
&+n^2 \frac{T\mathbb{R}_n \mathbb{R}_n''-T \mathbb{R}_n'{}^2+ \mathbb{R}_n \mathbb{R}_n'}{2 T \mathbb{R}_n}-n\frac{T \mathbb{R}_n'\left(T \mathbb{R}_n'-2  \mathbb{R}_n \right)}{4 T \mathbb{R}_n{}^2}\notag \\
&-\frac{T^2 \mathbb{R}_n'{}^2-2 T \mathbb{R}_n \mathbb{R}_n'+\mathbb{R}_n{}^2(\mathbb{R}_n{}^2-\alpha ^2-\beta ^2 +1)}{4 \mathbb{R}_n{}^3}\notag \\
   &+\frac{T(2 T \mathbb{R}_n \mathbb{R}_n'-T^2\mathbb{R}_n'{}^2+\beta ^2\mathbb{R}_n{}^2- \mathbb{R}_n{}^2)}{4 \mathbb{R}_n{}^3 \left(n \mathbb{R}_n-T\right)}=0.
\end{align}
Disregarding the derivative parts of \eqref{tR} and considering the first two terms with $n^3$ and $n^2$ yields
\begin{align}\label{}
\widehat{\mathbb{R}}_n(T)= \frac{1}{2}\left(2 n+1+\alpha +\beta \pm\sqrt{(\alpha +\beta +2 n+1)^2-8 T}\right).
\end{align}
Expanding into the  Taylor series as $n\rightarrow\infty$, we obtain two expressions of $\mathbb{R}_n(T)$:
\begin{align*}
\widehat{\mathbb{R}}_n(T)_1&=2n+\alpha +\beta +1-\frac{ T}{n}+\frac{(\alpha +\beta +1)T}{2n^2}+\mathcal{O}\left(\frac{1}{n^3}\right),\\
\widehat{\mathbb{R}}_n(T)_2&=\frac{T}{n}-\frac{(\alpha +\beta +1)T}{2 n^2}+\mathcal{O}\left(\frac{1}{n^3}\right).
\end{align*}
Next we assume that $\mathbb{R}_n(T)$ has the following expansion:
\begin{equation*}
  \mathbb{R}_n(T)=\sum_{j=0}^{\infty}a_j(T)n^{1-j},\quad n\rightarrow\infty.
\end{equation*}
Substituting the expression above into \eqref{tR} gives us \eqref{ZR} by comparing the corresponding coefficients on both sides.
\end{proof}
\end{proposition}

The proposition above is used to prove the following theorem.
\begin{theorem}
When $n\rightarrow\infty$, $t\rightarrow0^+$ and $T=t n$ is fixed, the monic polynomials $P_n(x)$ orthogonal with respect to the weight $w(x)=x^\alpha(1-x)^\beta{\rm e}^{-t x}$ on $[0,1]$ satisfy the confluent Heun equation
\begin{equation}\label{ZP2}
  P_n''(z)+\left(\frac{\widetilde{\gamma}}{z}+\frac{\widetilde{\delta}}{z-1}+\widetilde{\epsilon}\right)P_n'(z)+\frac{\widetilde{\alpha} z-\widetilde{q}}{z(z-1)}P_n(z)=0
\end{equation}
with parameters
$$
\widetilde{\gamma}=\alpha+1,\quad\widetilde{ \delta}=\beta+1,\quad \widetilde{\epsilon}=-T/n,\quad \widetilde{\alpha}=T,\quad \widetilde{q}=n(n+\alpha+\beta+1)+T/2.$$
\end{theorem}

\begin{proof}
Substituting \eqref{Zr}--\eqref{Zb} into \eqref{ZP1}, the coefficients of \eqref{ZP1}, $Q_n(z,t)$ and $S_n(z,t)$, are given in terms of $R_n(t)$ and $R'_n(t)$. In particular,
\begin{align*}
Q_n(z,t)=&\frac{\alpha+1}{z}+\frac{\beta+1}{z-1}-t-\frac{1}{z-R_n(t)/t},\\
 S_n(z,t)=&\frac{n z+\left[t R_n'(t)+\alpha t-(2n+1+\alpha+\beta+t)R_n(t)+R_n^2(t)\right]/(2t)}{z(z-1)(z-R_n(t)/t)}\\
 &+\frac{ t R_n'(t)+\alpha t-(2n+1+\alpha+\beta+t)R_n(t)+R_n^2(t)}{2z(z-1)}\\
  &+\frac{n(t z-1-\alpha -\beta -n)+t n (\beta +n)/\left(t-R_n(t)\right)}{z(z-1)}\\
 &+ \frac{(\alpha +\beta +2 n) \left[t R_n'(t)+\alpha t-(2n+1+\alpha+\beta+t)R_n(t)+R_n^2(t)\right]}{2z(z-1) \left(t-R_n(t)\right)}\\
 &+\frac{\left[t R_n'(t)+\alpha t-(2n+1+\alpha+\beta+t)R_n(t)+R_n^2(t)\right]^2 }{4z(z-1) R_n(t)\left(t-R_n(t)\right)}\\
 &-\frac{\alpha t\left[t R_n'(t)+\alpha t-(2n+1+\alpha+\beta+t)R_n(t)+R_n^2(t)\right] }{2z(z-1) R_n(t)\left(t-R_n(t)\right)}.
\end{align*}
Setting $T=n t$, the coefficients are further associated with $\mathbb{R}_n(T)$ and $\mathbb{R}'_n(T)$. From Proposition \ref{Zpropo2} we can substitute the asymptotic expression of $\mathbb{R}_n(T)$. Let $n$ tends to infinity. We obtain
\begin{align*}
   \mathbb{Q}_n(z,T)=&\frac{\alpha+1}{z}+\frac{\beta+1}{z-1}-\frac{T}{n}+\mathcal{O}\left(n^{-2}\right),\\
 \mathbb{S}_n(z,T)=&\frac{Tz-[n(n+\alpha+\beta+1)+T/2]}{z(z-1)}+\mathcal{O}\left(n^{-1}\right).
\end{align*}
Substituting above expressions into \eqref{ZP1}, we find the confluent Heun equation \eqref{ZP2}. Note that $P_n'(z)=n\sum _{j=0}^{n-1}b_jP_j(z),$ therefore, we take $\mathcal{O}\left(n^{-2}\right)$ for $\mathbb{Q}_n(z,T)$ and $\mathcal{O}\left(n^{-1}\right)$ for $\mathbb{S}_n(z,T)$.
\end{proof}

\begin{corollary}\label{Zc}
When $t=0$, the weight $w(t,x)=x^\alpha(1-x)^\beta{\rm e}^{-t x}$ reduces to the classical Jacobi weight $w(0,x)=x^\alpha(1-x)^\beta$ and the confluent Heun equation \eqref{ZP2} reduces to the hypergeometric differential equation (Jacobi differential equation)
\begin{equation}\label{ZP3}
  P_n''(z)+\left(\frac{\alpha+1}{z}+\frac{\beta+1}{z-1}\right)P_n'(z)-\frac{n(n+\alpha+\beta+1)}{z(z-1)}P_n(z)=0.
\end{equation}
\end{corollary}

\begin{proof}

In the case when $t=0$, we have  $T=0$. Then $\widetilde{\epsilon}=\widetilde{\alpha}=0$ and $\widetilde{q}=n(n+\alpha+\beta+1)$, which directly gives  \eqref{ZP3}.

Alternatively, we can use ladder operators to obtain the same result.

For  $w(0,x)=x^\alpha(1-x)^\beta$ we have $v(x)=-\alpha \ln x-\beta\ln (1-x)$ and $v'(x)=-\alpha/x-\beta/(x-1)$. From \eqref{A}--\eqref{B} we find
\begin{align*}
  A_n(z)&=\frac{1}{h_n}\int_0^1\;\left(\frac{\alpha}{yz}+\frac{\beta}{(y-1)(z-1)}\right)P^2_n(y)y^\alpha(1-y)^\beta\;{\rm d}y,\\
   B_n(z)&=\frac{1}{h_n}\int_0^1\;\left(\frac{\alpha}{yz}+\frac{\beta}{(y-1)(z-1)}\right)P_n(y)P_{n-1}(y)y^\alpha(1-y)^\beta\;{\rm d}y.
\end{align*}
Integrating by parts, it follows that
\begin{equation}
  A_n(z)=\frac{R_n}{z}-\frac{R_n}{z-1}.\label{ZA}
\end{equation}
Similarly one has,
\begin{equation}
  B_n(z)=\frac{r_n}{z}-\frac{n+r_n}{z-1}.\label{ZB}
\end{equation}
Here $R_n$ and $r_n$ are defined by
\begin{align*}
  &R_n=R_n(\alpha,\beta) := \frac{\alpha}{h_n}\int_0^1\;\frac{P_n ^2(y)}{y}y^\alpha(1-y)^\beta\;{\rm d}y,\\
  &r_n=r_n(\alpha,\beta) := \frac{\alpha}{h_{n-1}}\int_0^1\;\frac{P_n (y)P_{n-1}(y)}{y}y^\alpha(1-y)^\beta\;{\rm d}y.
\end{align*}

Substituting \eqref{ZA} and \eqref{ZB} into $(S_1)$ and equating residues of both sides of $(S_1)$ at $z=0$ and $z=1$ gives
\begin{align*}
  r_n+r_{n+1}&=-\alpha_nR_n+\alpha,\\
  -(2n+1+r_n+r_{n+1})&=-(1-\alpha_n)R_n+\beta.
\end{align*}
Obviously $R_n$ can immediately be obtained by adding two equalities above:
$$R_n=2n+1+\alpha+\beta.$$
Then we have
$$\sum_{j=0}^{n-1}R_j=n(n+\alpha+\beta).$$
Recall now \eqref{H},
\begin{align*}
  -\left(v'(z)+\frac{A'_n(z)}{A_n(z)}\right) &=\frac{\alpha+1}{z}+\frac{\beta+1}{z-1},  \\
  B'_n(z)-B_n(z)\frac{A'_n(z)}{A_n(z)}+\sum_{j=0}^{n-1}A_j(z)&=-\frac{n(n+\alpha+\beta+1)}{z(z-1)},
\end{align*}
which produces \eqref{ZP3}.

Actually, we can obtain other relations for $r_n$, recurrence coefficients $\alpha_n$ and $\beta_n$ by using $(S_2)$ and $(S_2')$.  See similar calculations for the classical Jacobi weight $(1-x)^\alpha(1+x)^\beta$ in  Chen and Ismail \cite{ChenIsmail2005}, where  explicit expressions for $r_n$, $\alpha_n$, $\beta_n$ and also   explicit expressions for the polynomials were obtained.
\end{proof}

\begin{rem} Combining \eqref{Zr}--\eqref{Zb} with Proposition \ref{Zpropo2}, and sending $n\rightarrow \infty$, $t\rightarrow 0^+$ and keeping  $T=nt$  fixed, we find the following asymptotic expressions for  $r_n(T/n)$, $\mathbb{\beta}_n(T/n)$ and $\sum_{j=0}^{n-1}\mathbb{R}_j(T)$:
\begin{align}
r_n(T/n)=&-\frac{n}{2}+\frac{\alpha -\beta }{4}-\frac{T +\beta ^2-\alpha ^2}{8 n}+\frac{(\alpha -\beta ) (\alpha +\beta )^2}{16 n^2}\notag\\
&~~~+\frac{T \left(2 \alpha ^2+2 \beta ^2-1\right)-(\alpha -\beta ) (\alpha +\beta )^3}{32 n^3}+\mathcal{O}\left(\frac{1}{n^3}\right), \\
\beta_n(T/n)=&\frac{1}{16}+\frac{1-2 \alpha ^2-2 \beta ^2}{64 n^2}+\frac{\eta_n}{256 n^3 T (\alpha +\beta +1)}+\mathcal{O}\left(\frac{1}{n^4}\right),\label{Zb1}\\
\sum_{j=0}^{n-1}\mathbb{R}_j(T)=&n(n+\alpha +\beta )+\frac{T }{2}+\frac{(\beta   -\alpha)  T }{4 n}+\frac{(2 (\alpha ^2 - \beta ^2)  +T)T}{16 n^2}\notag\\
&~~~~~~~~~~~~~~~~~~~~~~~~~~-\frac{T (\alpha -\beta ) (\alpha +\beta )^2}{16 n^3}+\mathcal{O}\left(\frac{1}{n^4}\right)\label{ZSR1},
\end{align}
where
\begin{align*}
  \eta_n=&\left(\alpha ^2-\beta ^2\right) \left[3 T^2+16+5 \alpha ^4+20 \alpha ^3 (\beta +1)+10 \alpha ^2 (3 \beta  (\beta +2)+4)\right.\notag\\
&\left.+20 \alpha  (\beta +1) (\beta  (\beta +2)+2)+5 \beta  (\beta +2) (\beta  (\beta +2)+4)\right]\notag\\&-2 T \left[\alpha ^4+4 \alpha ^3 (\beta +2)+\alpha ^2 \left(6 \beta ^2+8 \beta +9\right)+2 \alpha  (\beta +2) \left(2 \beta ^2-1\right)\right.\notag\\
&\left.+\beta  (\beta  (\beta  (\beta +8)+9)-4)-5\right].
\end{align*}

\end{rem}

Next  we consider the second method (Dyson's Coulomb fluid method) to obtain asymptotic expression of the recurrence coefficient $\alpha_n(t)$.  Using relation
\begin{equation}\label{}
R_n(t)=2n+1+\alpha+\beta+t-t\alpha_n(t),
\end{equation}
see \cite{ZGCZ}, we then can deduce the asymptotic expression for $R_n(t)$. Using this method we can  verify the accuracy of $R_n(t)$ which was obtained in Proposition \ref{Zpropo2}.

\begin{proposition}
Sending $n\rightarrow\infty$, we obtain the following asymptotic expressions of the recurrence coefficients:
 \begin{align}
  \alpha_n(T/n)  \sim &\frac{1}{2}+\frac{\alpha^2-\beta^2}{8n^2}-\frac{(\alpha -\beta ) (\alpha +\beta )^2}{8 n^3}\notag\\
  &~~~~~~~~~~~~~~~+\frac{3 (\alpha -\beta ) (\alpha +\beta )^3-2 T \left(\alpha ^2+\beta ^2\right)}{32 n^4}+\mathcal{O}\left(\frac{1}{n^3}\right),\label{Za1}\\
  \beta_n(T/n)\sim&\frac{1}{16}-\frac{\alpha ^2+\beta ^2}{32 n^2}+\frac{(\alpha +\beta ) \left(\alpha ^2+\beta ^2\right)}{32 n^3}\notag\\
  &~~~-\frac{ (\alpha +\beta )^2 \left(5 \alpha ^2+2 \alpha  \beta +5 \beta ^2\right)-4 T (\alpha^2-\beta^2  )}{256 n^4}+\mathcal{O}\left(\frac{1}{n^5}\right).\label{Zb2}
\end{align}
\end{proposition}
\begin{proof}

For $w(x)=x^\alpha(1-x)^\beta{\rm e}^{-t x}$ we have
 $$ v(x)=-\alpha\ln x -\beta\ln (1-x)+t x,\quad v'(x)=-\frac{\alpha}{ x} -\frac{\beta}{x-1}+t.$$

 Substituting $v'(x)$ into  \eqref{supp2}, by using formulas \eqref{int1}, \eqref{int2} and \eqref{int3} in Appendix 1,  we obtain the following two algebraic equations:
\begin{align}
t +\frac{\beta}{\sqrt{(1-a)(1-b)}}-\frac{\alpha}{\sqrt{ab}}&=0,\label{Zagb1}\\
2n+\alpha+\beta-\frac{\beta}{\sqrt{(1-a)(1-b)}}-\frac{(a+b)t}{2}&=0.\label{Zagb2}
\end{align}
Adding  two algebraic equations above, we obtain
\begin{align}
2n+\alpha +\beta +t -\frac{\alpha}{\sqrt{ab}}-\frac{(a+b)t}{2}=0.\label{Zagb3}
\end{align}

Take  $X(t):=\frac{1}{\sqrt{ab}}$  with $X(0)=(2n+\alpha+\beta)/\alpha$. Let   $t=0$ in \eqref{Zagb3}. Solving for $(a+b)/2$  from \eqref{Zagb3}, we find
\begin{equation}\label{Zagb4}
\frac{a+b}{2}=\frac{2n+\alpha+\beta+t-\alpha X(t)}{t}.
\end{equation}
Substituting  $(a+b)/2$ into the square root $\sqrt{(1-a)(1-b)}$ of \eqref{Zagb1} and using $X(t)$ to replace $1/\sqrt{ab}$, we get
\begin{equation*}
  t +\beta/\sqrt{1-\frac{2}{t}(2n+\alpha+\beta+t-\alpha X(t))+\frac{1}{X(t)^2}}-\alpha X(t)=0.
\end{equation*}
After some simple calculations we  find  that $X(t)$ satisfies the quintic equation
\begin{align}\label{ZX}
2 \alpha ^3 X(t)^5&- \alpha^2\left(2 \alpha +2 \beta +5 t +4  n\right)X(t)^4+ 4 \alpha  t  (2n+\alpha +\beta +t )X(t)^3 \notag\\
&- t\left( 4 t  n -\alpha ^2+2 \alpha  t +(\beta +t )^2\right)X(t)^2-2 \alpha  t ^2 X(t)+t ^3=0
\end{align}
 with  $X(0)=(2n+\alpha+\beta)/\alpha$.

Let $n\rightarrow \infty$, $t\rightarrow 0^+$ and $T=nt$ be fixed. We can find an equivalent  quintic equation in terms of $T$. Consider the fist two terms of $n$ and $n^0$ of this new quintic equation,
\begin{equation*}
  \left(8 \alpha  T \widetilde{\mathbb{X}}(T)^3+2 \alpha ^3 \widetilde{\mathbb{X}}(T)^5-2 \alpha ^3 \widetilde{\mathbb{X}}(T)^4-2 \alpha ^2 \beta  \widetilde{\mathbb{X}}(T)^4\right)-4 n \left(\alpha ^2 \widetilde{\mathbb{X}}(T)^4\right)=0.
\end{equation*}
Solving the  equation above we  obtain two nonzero solutions,
\begin{align*}
 \widetilde{\mathbb{X}}(T)=\frac{2n+\alpha +  \beta \pm\sqrt{\left(2n+\alpha +  \beta \right)^2-16  T}}{2 \alpha }.
\end{align*}
Taking  the Taylor series for  large $n$, we  obtain
\begin{align*}
 \widetilde{\mathbb{X}}_1(T)&=\frac{2 T}{\alpha n}-\frac{(\alpha+\beta)T}{\alpha n^2}+\mathcal{O}\left(\frac{1}{n^3}\right),\\
  \widetilde{\mathbb{X}}_2(T)&=\frac{2 n+\alpha+\beta}{\alpha }-\frac{2 T}{\alpha n}+\frac{(\alpha+\beta)T}{\alpha n^2}+\mathcal{O}\left(\frac{1}{n^3}\right).
\end{align*}
Assuming that  $\mathbb{X}(T)$ has the form
$$\mathbb{X}(T)=\sum_{j=0}^{\infty}b_j(T)n^{1-j},\quad n\rightarrow\infty,$$
and substituting the expression above into \eqref{ZX} with $X(0)=(2n+\alpha+\beta)/\alpha$, we obtain when $n\rightarrow\infty$
\begin{align}
\mathbb{X}(T)=\frac{2n+\alpha+\beta}{\alpha}&+\frac{T}{2\alpha n}-\frac{(\alpha^2-\beta^2)T}{8\alpha n^3}+\frac{T (\alpha -\beta ) (\alpha +\beta )^2}{8 \alpha  n^4}\notag\\
&+\frac{T \left(2 T \left(\alpha ^2+\beta ^2\right)-3 (\alpha -\beta ) (\alpha +\beta )^3\right)}{32 \alpha  n^5}+\mathcal{O}\left(\frac{1}{n^4}\right).
\end{align}
Setting $T=nt$ in \eqref{Zagb4} we obtain
 \begin{align*}
  \alpha_n(T/n)  \sim &\frac{a+b}{2}=\frac{2n+\alpha+\beta+T/n-\alpha \mathbb{X}(T)}{T/n}
  \end{align*}
and
  \begin{align*}
  \beta_n(T/n)\sim&\frac{[(b+a)/2]^2-ab}{4}=\frac{1}{4}\left[\left(\frac{2n+\alpha+\beta+T/n-\alpha \mathbb{X}(T)}{T/n}\right)^2-\frac{1}{\mathbb{X}(T)^2}\right]\notag.
\end{align*}
Substituting $\mathbb{X}(T)$ into  the expressions above and sending  $n$ to infinity, we obtain \eqref{Za1} and \eqref{Zb2}.

\end{proof}

\begin{rem}
 From Dyson's Coulomb fluid approximation theory we obtain the same order asymptotic expression of
\begin{align*}
\mathbb{R}_n(T)=&2n+1+\alpha+\beta+\frac{T}{n}-\frac{T\alpha_n(T/n)}{n}\\
\sim &2n+1+\alpha+\beta+\frac{T}{2n}-\frac{T(\alpha^2-\beta^2)}{8n^3}+\mathcal{O}\left(\frac{1}{n^4}\right).
\end{align*}
\end{rem}
However, if we compare \eqref{Zb2} with \eqref{Zb1}, they are not exactly the same. As we mentioned in Introduction, the Coulomb fluid method is suitable for sufficiently large $n$. We see that up to   the order of $\mathcal{O}\left(\frac{1}{n}\right)$ they are equal.

Finally, if we consider  equation \eqref{ZP1} depending on  functions $R_n(t)$ and $R_n'(t)$ satisfying equation (\ref{nonlinear eq Rn section 2.1}) (without any reference to orthogonal polynomials), we can obtain that  it is an equation for the derivative of the confluent Heun function in a special case.

\begin{proposition}
If $R_n(t)$ satisfies the Riccati equation
$$t R_n'(t)= R_n(t)^2-(\alpha +\beta +t-1) R_n(t)+\alpha  t$$
with solution
\begin{equation}\label{}
R_n(t)= \frac{t \left(C_1 \beta   U(\beta +1,\alpha+\beta +1,t)+L_{-\beta-1}^{\alpha +\beta}(t)\right)}{C_1 U(\beta ,\alpha +\beta,t)+L_{-\beta }^{\alpha +\beta -1}(t)}+t,
\end{equation}
then equation \eqref{ZP1} reduces to the equation \eqref{dC}
\begin{align}
 {P}_n''(z)&+ \left(\frac{\widetilde{\gamma} +1}{z}+\frac{\widetilde{\delta} +1}{z-1}+\widetilde{\epsilon} -\frac{\widetilde{\alpha} }{\widetilde{\alpha}  z-\widetilde{q}}\right){P}_n'(z)\notag\\
 &~~~~~~~~~~~~~+\frac{ (\widetilde{\alpha} +\widetilde{\epsilon} ) \left(\widetilde{\alpha}  z^2-2 \widetilde{q} z\right)+\widetilde{\alpha}  \widetilde{\gamma} +\widetilde{q}^2-\widetilde{q} (\widetilde{\gamma} +\widetilde{\delta} -\widetilde{\epsilon} )}{z (z-1) (\widetilde{\alpha } z-\widetilde{q})}{P}_n(z)=0,
\end{align}
with parameters
$$\widetilde{\gamma} = \alpha ,\;\;\widetilde{\delta} = \beta ,\;\; \widetilde{\epsilon} = -t,\;\; \widetilde{\alpha}=(n+1) t,\;\; \widetilde{q}= (n+1)R_n(t).$$
In the special case, when the  constant $C_1=0$, we have
$$R_n(t)=\frac{\alpha  t \, M(\beta ;\alpha +\beta +1;t)}{(\alpha+\beta)\, M(\beta ;\alpha +\beta ;t)},$$
where $M(a;b;x)$, $U(a;b;x)$ is  Kummer function of first and second kind \cite[Sec.~13.2]{DLMF} and $L_n^a(x)$ is generalized Laguerre polynomial \cite[Sec.~18.1]{DLMF}. The Riccati equation satisfies  (\ref{nonlinear eq Rn section 2.1}) when $n=-1$.
\end{proposition}

\subsection{$x^\alpha(1-x)^\beta{\rm e}^{-t/ x},\; x\in(0,1],\alpha,\beta,t>0$ }
The weight  $x^\alpha(1-x)^\beta{\rm e}^{-t/ x}$ on $(0,1]$ was  considered in Chen, Dai\cite{CCF} and Chen \cite{CD}.
The second order differential equation for  $P_n(z)$  is as follows:
\begin{align}\label{CD1}
P_n''(z)+Q_n(z,t)P_n'(z)+ S_n(z,t)P_n(z)=0,
\end{align}
where
\begin{align*}
Q_n(z,t)=&\frac{t}{z^2}-\frac{1}{z-\tilde{R}_n(t)/[\tilde{R}_n(t)-R_n(t)]}+\frac{\alpha +2}{z}+\frac{\beta +1}{z-1},\\
S_n(z,t)=&\frac{\tilde{R}_n(t) \left[n (z-1)^2-r_n(t)\right]+R_n(t) \tilde{r}_n(t)-n R_n(t) z^2}{(z-1) z^2 \left[z(R_n(t)- \tilde{R}_n(t))+\tilde{R}_n(t)\right]} \notag\\&+\frac{\sum _{j=0}^{n-1} \tilde{R}_j(t)}{z^2}-\frac{\sum _{j=0}^{n-1} R_j(t)}{z (z-1)}.
\end{align*}
See Chen and Dai \cite{CD} for the  definitions of $\tilde{r}_n(t)$, $r_n(t)$, $\tilde{R}_n(t)$ and $R_n(t)$. They satisfy the following relations  (see \cite {CD} for details):
\begin{align}
  \tilde{R}_n(t)=&R_n(t)-(\alpha +\beta +2 n+1),\label{CDR}   \\
r_n(t)=&\frac{R_n(t)-\beta -\alpha _n(t)R_n(t)}{2} +\frac{t R'_n(t)-\alpha _n(t)R_n(t) }{2 (\alpha +\beta +2 n+1)},\label{CDrm}\\
  \tilde{r}_n(t)=&\frac{R_n(t)-\beta -\alpha _n(t)R_n(t)}{2} +\frac{t R'_n(t)-\alpha _n(t)R_n(t) }{2 (\alpha +\beta +2 n+1)}\notag\\&~~~~~~~~~~~~~~~~~+\frac{\beta +t+(\alpha +\beta +2 n+2) \alpha _n(t)-R_n(t)}{2}, \label{CDr}
  \\
  \sum_{j=0}^{n-1} \tilde{R}_j(t)=&n(t-\alpha-n)-(2n+\alpha+\beta)(\tilde{r}_n-r_n),\label{CDSRt}\\
  \sum_{j=0}^{n-1} R_j(t)=&\sum_{j=0}^{n-1} \tilde{R}_j(t)+n(n+\alpha+\beta).\label{CDSR}
\end{align}

 The equation \eqref{CD1} is neither Heun equation nor equation for the derivative of the Heun function. However, we will show that it is   a confluent Heun equation using the  asymptotic behaviour of its coefficients.  To obtain asymptotic behavior of $\alpha_n(t)$ we will use  Dyson's Coulomb fluid method. For $R_n(t)$, we will use  the formulas obtained by Chen and Dai \cite{CD}, Chen and Chen\cite{CCF}. Here we just show a brief statement, for further details and proof see \cite[\S2]{CCF}.  This will   further be used to  investigate the second order linear  ordinary differential equation \eqref{CD1}.

\begin{proposition}\cite[\S2]{CCF}.\label{CF}
Define
$$f(t,\alpha,\beta):=n^2\left(\frac{R_n(t)}{2n+1+\alpha+\beta}-1\right).$$
Let  $t\rightarrow0^+$, $n\rightarrow\infty$ and $ T:=2n^2t$ be fixed. If
$$F(T,\alpha,\beta):=\lim_{n\rightarrow\infty}f\left(\frac{T}{2n^2},\alpha,\beta\right),$$
then $F(T,\alpha,\beta)$ satisfies
\begin{equation}\label{F}
F''=\frac{F'{}^2}{F}-\frac{F'}{T}+\frac{2F^2}{T^2}+\frac{\alpha}{2T}-\frac{1}{4F}
\end{equation}
with $F(0,\alpha,\beta)=0,\;F'(0,\alpha,\beta)=1/(2\alpha)$. Equation (\ref{F}) is the third Painlev\'e equation $P_{III'}(8,2\alpha,0,-1)$.

Moreover, for $\alpha\neq \mathbb{Z}$, the following expansion holds:
\begin{equation}
F(T,\alpha,\beta)=\frac{T}{2\alpha}-\frac{T^2}{2\alpha^2(\alpha^2-1)}+\frac{3T^3}{2\alpha^3(\alpha^2-4)(\alpha^2-1)}+\mathcal{O}\left(T^4\right).
\end{equation}
\end{proposition}

 For convenience we use hollow symbol to define a function of  variable $T$ by $\mathbb{R}_n(T)=R_n(T/(2n^2))$.

\begin{proposition}
For $T:=2n^2t$ fixed, $t\rightarrow0^+$, $n\rightarrow\infty$ and for small $T$  we have
\begin{align}
&\mathbb{R}_n(T)=(2n+\alpha+\beta+1)\left[1+\frac{T}{2n^2\alpha}
-\frac{T^2}{2n^2\alpha^2(\alpha^2-1)}+\mathcal{O}\left(T^3\right)\right],\;\alpha\neq\mathbb{Z},\label{CCFr}\\
&\alpha_n(T/(2n^2)) \sim\frac{1}{2}+\frac{\alpha ^2-\beta ^2}{8 n^2}+\frac{T}{4 \alpha  n^2}-\frac{3 T^2}{8 \alpha ^4 n^2}+\mathcal{O}\left(T^3\right),\label{CCFa}\\
&\beta_n(T/(2n^2))\sim\frac{1}{16}-\frac{\alpha ^2+\beta ^2}{32 n^2}-\frac{T}{16 \alpha  n^2}+\frac{3T^2}{32 \alpha ^4 n^2}+\mathcal{O}\left(T^3\right).\label{CCFb}
\end{align}
\end{proposition}

\begin{proof}

From Proposition \ref{CF} it follows that
\begin{align*}
\mathbb{R}_n(T)=&(2n+1+\alpha+\beta)\left[1+\frac{F(T,\alpha,\beta)}{n^2}\right],
\end{align*}
which gives  \eqref{CCFr}.

Next  we will use  Dyson's Coulomb fluid method to obtain the asymptotic behavior of  $\alpha_n(t)$. For the Pollaczek--Jacobi type weight $x^\alpha(1-x)^\beta{\rm e}^{-t/ x}$ we have
 $$v(x)=-{\rm ln}\, w(x)=\frac{t}{x}-\alpha {\rm ln}\, x-\beta {\rm ln}\,(1-x)\;\;\text{and}\;\;v'(x)=-\frac{t}{x^2}-\frac{\alpha}{x}-\frac{\beta}{x-1}.$$

Substituting  $v'(x)$ into \eqref{supp2} and combing with formulas \eqref{int1}, \eqref{int3} and \eqref{int4}, we obtain two algebraic equations with respect to $a$ and $b$
\begin{align}\label{}
 &\frac{a+b}{2(ab)^{3/2}}t+\frac{\alpha }{\sqrt{ab}}-\frac{\beta}{\sqrt{(1-a)(1-b)}}=0,\label{Calgb1}\\
 &\frac{t}{\sqrt{ab}}-\frac{\beta}{\sqrt{(1-a)(1-b)}}+2n+\alpha+\beta=0.\label{Calgb2}
\end{align}

Let $Y(t):=1/(\sqrt{ab})$. Subtracting  \eqref{Calgb1} from \eqref{Calgb2} we obtain
  \begin{equation}\label{Calgb3}
 \frac{a+b}{2}=\frac{2n+\alpha+\beta-(\alpha-t)Y(t)}{tY(t)^3}.
  \end{equation}
Substituting the equality above into \eqref{Calgb2} yields
\begin{equation*}
  tY(t)-\beta/\sqrt{1-\frac{2}{tY(t)^3}\left[2n+\alpha+\beta-(\alpha-t)Y(t)\right]+\frac{1}{Y(t)^2}}+2n+\alpha+\beta=0.
\end{equation*}
 Setting $T=2n^2 t$ and $\mathbb{Y}(T)=Y(T/(2n^2))$,  after some simple calculations we obtain the quintic equation
\begin{align}\label{CDY}
\frac{T\mathbb{Y}(T)^3}{4n^2}-\frac{\beta^2\mathbb{Y}(T)^3T/(2n^2)}{2( \mathbb{Y}(T)T/(2n^2)+2n+\alpha+\beta)^2}&+\alpha \mathbb{Y}(T)\notag \\
-\frac{\mathbb{Y}(T)T/(2n^2)}{2}&-(2n+\alpha+\beta)=0
\end{align}
with $\mathbb{Y}(0)=(2n+\alpha+\beta)/\alpha$ (by setting $T=0$ in the equation  above).

Consider the terms at $T^0$ and $T$,
\begin{equation*}
  T \left(\frac{\widetilde{\mathbb{Y}}(T)^3}{4 n^2}-\frac{\beta ^2 \widetilde{\mathbb{Y}}(T)^3}{4 n^2 (\alpha +\beta +2 n)^2}-\frac{\widetilde{\mathbb{Y}}(T)}{4 n^2}\right)+(-\alpha -\beta -2 n+\alpha  \widetilde{\mathbb{Y}}(T))=0.
\end{equation*}
For large $n$, the term $-\beta ^2 \widetilde{\mathbb{Y}}(T)^3/(4 n^2 (\alpha +\beta +2 n)^2)$ does not affect the form of the solution. So solving the equation above  without this term  we obtain
$$\widetilde{\mathbb{Y}}(T)=\frac{\alpha +\beta +2 n}{\alpha }-\frac{T \left((\beta +2 n) \left(2 \alpha ^2+3 \alpha  \beta +(\beta+2n) ^2+6 \alpha  n\right)\right)}{4 \left(\alpha ^4 n^2\right)}+\mathcal{O}\left(T^2\right).$$
Then we assume that $\mathbb{Y}(T)$ has the following  form:
 $$\mathbb{Y}(T)=\sum_{j=0}^{\infty}b_jT^{j},\quad T\rightarrow0^+.$$
 Substituting the expression above into \eqref{CDY} with $\mathbb{Y}(0)=(2n+\alpha+\beta)/\alpha$, we have
\begin{align*}
\mathbb{Y}(T)\sim&\frac{2 n+\alpha +\beta }{\alpha }-\frac{(\alpha +\beta +2 n) [  \alpha\beta+ 2 n (\alpha+\beta +n)]}{2 \alpha ^4 n^2}T\\
&+\frac{ (\alpha +\beta +2 n)\eta_{n1}T^2 }{4 \alpha ^7 n^4}+\frac{ (\alpha +\beta +2 n)\eta_{n2}T^3 }{8 \alpha ^{10} n^6}+\mathcal{O}\left(T^4\right),\quad T\rightarrow0^+,
\end{align*}
where
\begin{align*}
\eta_{n1}&=12 n^4+24 n^3 (\alpha +\beta )+2 n^2 \left(7 \alpha ^2+18 \alpha  \beta +6 \beta ^2\right)+\alpha ^2 \beta  (\alpha +2 \beta )\\&~~~+2 \alpha  n (\alpha +\beta ) (\alpha +6 \beta ),\\
\eta_{n2}&=\beta ^3 (\alpha +2 n) \left(5 \alpha ^2+48 n^2+48 \alpha  n\right)+\beta  (\alpha +2 n) \left(\alpha ^2+6 n^2+6 \alpha  n\right) \\&~~~\cdot\left(\alpha ^2+24 n^2+24 \alpha  n\right)+\beta ^2 \left(5 \alpha ^4+288 n^4+576 \alpha  n^3+376 \alpha ^2 n^2+88 \alpha ^3 n\right)\\&~~~+2 n (\alpha +n) \left(\alpha ^4+48 n^4+96 \alpha  n^3+63 \alpha ^2 n^2+15 \alpha ^3 n\right).
\end{align*}

Substituting $\mathbb{Y}(T)$ into \eqref{Calgb3} with $t=T/(2n^2)$ we get
 $$\frac{a+b}{2}=\frac{2n+\alpha+\beta-(\alpha-T/2n^2)\mathbb{Y}(T)}{  T/(2n^2)  \mathbb{Y}(T)^3},$$
which gives \eqref{CCFa} for small $T$ and as $n\to \infty$. Since
\begin{align*}
 \beta_n(T/(2n^2))&\sim \frac{[(a+b)/2]^2-ab}{4}\\&=\frac{1}{4}\left[\left(\frac{2n+\alpha+\beta+(T/2n^2-\alpha)\mathbb{Y}(T)}{\mathbb{Y}(T)^3T/2n^2}\right)^2-\frac{1}{\mathbb{Y}(T)^2}\right],
\end{align*}
then  for small $T$ we deduce \eqref{CCFb}.

\end{proof}

Once the asymptotic formulas for $\mathbb{R}_n(T),\; \alpha_n(T/2n^2)$ are obtained, we have the following theorem.

\begin{theorem}
Let  $n\rightarrow\infty,\; t\rightarrow0^+$ and $T=2n^2t$ be fixed. For small $T$, the orthogonal polynomials $\widehat{P}_n(u)$ satisfy the confluent Heun equation,
\begin{align}\label{CD2}
 \widehat{P}_n''(u)+\left(\frac{\widetilde{\gamma}}{u}+\frac{\widetilde{\delta}}{u-1}+\widetilde{\epsilon}\right)\widehat{P}_n'(u)+\frac{\widetilde{\alpha} u-\widetilde{q}}{u(u-1)}\widehat{P}_n(u)=0
\end{align}
with parameters
$$\widetilde{\gamma}=-\alpha-\beta+2\lambda,\quad\widetilde{ \delta}=\beta+1,\quad \widetilde{\epsilon}=-s,\quad \widetilde{\alpha}=-s\lambda,\quad \widetilde{q}=\lambda  (\alpha -\lambda -s)-a.$$
Here   $\widehat{P}_n(u):=u^{-\lambda} P_n(1/u),$  $u:=1/z.$  The parameters $a, b, s, \lambda$ are given by
$$s=\frac{ [(\alpha -1) \alpha  (\alpha +1)^2-T]T}{2 n^2 \alpha ^2 \left(\alpha ^2-1\right) },\qquad a=\frac{\left(\alpha ^2+3\right) T^2-2\alpha^3(\alpha^2-1)T}{4 \alpha ^4 \left(\alpha ^2-1\right) },$$
$$\;b=n (n+ \alpha +\beta +1),\qquad \lambda=\frac{\alpha+\beta +1 \pm\sqrt{(\alpha +\beta +1)^2+4 b}}{2} .$$
\end{theorem}

\begin{proof}

Substituting \eqref{CDR}--\eqref{CDSR} into \eqref{CD1} and taking $T=2n^2t$ we find that the coefficients of the second order linear ordinary differential equation  \eqref{CD1}, $Q_n(z,T)$ and $S_n(z,T)$, are given in terms of  $\mathbb{R}_n(T)$, $\mathbb{R}'_n(T)$ and $\alpha_n(T/(2n^2))$.

Sending  $T\rightarrow0^+$, $n\to\infty$ and combining with the asymptotic expressions \eqref{CCFr} and \eqref{CCFa} yields
\begin{align*}
 \mathbb{ Q}_n(z,T) &=\frac{s}{z^2}+\frac{\alpha +1}{z}+\frac{\beta +1}{z-1}+\mathcal{O}\left(T^3\right),\quad T\rightarrow0^+, \\
  \mathbb{ S}_n(z,T) &=-\frac{b z+a}{(z-1) z^2}+\mathcal{O}\left(T^3\right),~~~~~~~~~~~\quad T\rightarrow0^+.
\end{align*}
Namely, we have
 \begin{align}\label{CD3}
P_n''(z)&+\left(\frac{s}{z^2}+\frac{\alpha +1}{z}+\frac{\beta +1}{z-1}\right)P_n'(z)-\frac{b z+a}{(z-1) z^2}P_n(z)=0.
\end{align}
Let $u=1/z$, then $$\widehat{P}_n(u):=u^{-\lambda} P_n(1/u)=z^{\lambda}P_n(z) $$ satisfy the confluent Heun equation \eqref{CD2}.
\end{proof}


\begin{corollary}
 The confluent Heun equation \eqref{CD3}  
 with $t=0$ reduces to the same Jacobi differential equation as in Corollary~\ref{Zc}:
 \begin{align}\label{CD4}
P_n''(z)&+\left(\frac{\alpha +1}{z}+\frac{\beta +1}{z-1}\right)P_n'(z)-\frac{n (n+ \alpha +\beta +1)}{(z-1) z}P_n(z)=0.
\end{align}
\end{corollary}
 \begin{proof}
 When $t=0$, we have $T=0$, $a=s=0$, and equation \eqref{CD3} reduces to  \eqref{CD4}.

 Actually, the weight  $w(t,x)=x^\alpha(1-x)^\beta{\rm e}^{-t/ x}$ is the deformed Jacobi weight. When $t=0$, the weight reduces to the  classical Jacobi weight $w(0,x)=x^\alpha(1-x)^\beta$. See the proof of  Corollary \ref{Zc}  for ladder operators approach to deduce the Jacobi differential equation.
 \end{proof}

\subsection{$(1-x^2)^\alpha(1-k^2x^2)^\beta, \;x\in[-1,1], \alpha,\,\beta>0, k^2\in(0,1)$}
 The  weight $(1-x^2)^\alpha(1-k^2x^2)^\beta$ is the generalization of the weight function  $\big[(1-x^2)(1-k^2x^2)\big]^{-1/2}$ studied by C. J. Rees in 1945  \cite{R1945}.   It has strong relation with the famous string theory, see Basor, Chen and Haq \cite{BasorChen}. In \cite{BasorChen} the asymptotic expressions of recurrence coefficients $\beta_n(k^2)$ and of the second coefficients $p(n)$ of monic polynomials $P_n(z)$  were obtained.  The large $n$ asymptotics of Hankel determinants was also studied.

The second order linear ordinary differential equation for $P_n(z)$, associated with the weight $(1-x^2)^\alpha(1-k^2x^2)^\beta$, reads,
\begin{equation}\label{BC1}
P''_n(z)+\left(\frac{X'(z)}{2X(z)}-\frac{M'_n(z)}{M_n(z)}\right)P'_n(z)+\left(\frac{L_n(z)M'_n(z)}{Y(z)M_n(z)}+\frac{U_n(z)}{Y(z)}\right)P_n(z)=0,
\end{equation}
where
\begin{align*}
X(z):=&(1-z^2)^{2\alpha+2}(1-k^2z^2)^{2\beta+2},\\
Y(z):=&(1-z^2)(1-k^2z^2),\\
M_n(z):=&-2\left(n+1/2+\alpha+\beta\right)k^2z^2-C_n,\\
L_n(z):=&z\left[nk^2z^2-n(k^2+1)+2k^2(n+1/2+\alpha+\beta)\beta_n-2k^2p(n)\right],\\
U_n(z):=&-k^2z^2n(n+2\alpha+2\beta+3)+2k^2(2n+1+2\alpha+2\beta)(p(n)-\beta_n)\\
&+nk^2(n+1+2\beta)+n(n+1+2\alpha)
\end{align*}
and
$$C_n=2k^2(n+3/2+\alpha+\beta)(\beta_n+\beta_{n+1})-2\left[(n+\beta+1/2)k^2+n+\alpha+1/2\right]-4k^2p(n).$$
The coefficients of  the second order differential equation  (\ref{BC1}) depend on  $p(n)$ and $\beta_n(k^2)$ after  substituting equalities above into it. From  Sections 2.1 and 2.2 we know that in order to reduce   equation (\ref{BC1}) to  the Heun equation we first need  to know the asymptotic expressions of $p(n)$ and $\beta_n(k^2)$.

\begin{theorem}\label{Jacobi1}
Let $P_n(x)$ be the monic orthogonal polynomials with respect to the weight $(1-x^2)^\alpha(1-k^2x^2)^\beta$, $n\rightarrow \infty,\; k^2\rightarrow0^+, \;\beta\rightarrow\infty$ and $k^2\beta=t$ be fixed.  Under these assumptions the weight reduces to  $(1-x^2)^\alpha{\rm e}^{-t x^2}$. Then if  $\widehat{P}_n(u)=P_n(z^2)$,
then $\widehat{P}_n(u)$ satisfies the confluent Heun equation
\begin{equation}\label{BC4}
\widehat{P}_n''(u)+\Big(\frac{\widetilde{\gamma}}{u}+\frac{\widetilde{\delta}}{u-1}+\widetilde{\epsilon}\Big)\widehat{P}_n'(u) +\left(\frac{\widetilde{\alpha} u-\widetilde{q}}{u(u-1)}\right)\widehat{P}_n(u)=0
\end{equation}
with parameters
$$\widetilde{\gamma}=1/2,\quad \widetilde{\delta}=\alpha+1,\quad \widetilde{\epsilon}=-t,\quad \widetilde{a}=nt/2,\quad \widetilde{q}=n(n+2\alpha+t+1)/4.$$
\end{theorem}
\begin{proof}

From Kuijlaars, McLaughlin, Assche and Vanlessen \cite{Kui2004}, and Basor, Chen and Haq \cite{BasorChen}, the asymptotic   expressions as $n\rightarrow\infty$ of $p(n),\;\beta_n(k^2)$  are known:
\begin{align*}
&\beta_n(k^2)=\frac{1}{4}-\frac{4\alpha^2-1}{16n^2}+\frac{4\alpha^2-1}{8n^3}\left(\alpha+\beta-\frac{\beta}{\sqrt{1-k^2}}\right)+\mathcal{O}\left(\frac{1}{n^4}\right),\\
&p(n):=-\frac{n}{4}+\frac{\alpha-\beta}{4}+\frac{1}{8}+\frac{\beta(1-\sqrt{1-k^2})}{2k^2}-\frac{4\alpha^2-1}{16n^3}\left(\alpha+\beta-\frac{1}{2}-\frac{\beta}{\sqrt{1-k^2}}\right)^2\\
&~~~~~~~~~~~~~~~~-\frac{4\alpha^2-1}{16n}+\frac{4\alpha^2-1}{16n^2}\left(\alpha+\beta-\frac{1}{2}-\frac{\beta}{\sqrt{1-k^2}}\right)+\mathcal{O}\left(\frac{1}{n^4}\right),
\end{align*}

Since the coefficients of \eqref{BC1} are represented by $p(n),\;\beta_n(k^2)$,  we have the following equation as  $n\rightarrow\infty$:
\begin{align}\label{BC2}
P''_n(z)&+\left(\frac{\beta  k}{k z-1}+\frac{\beta  k}{k z+1}+\frac{\alpha +1}{z-1}+\frac{\alpha +1}{z+1}\right)P'_n(z)+\left(\frac{n^2}{1-z^2}\right.\notag\\
&\left.+\frac{n \left(2 \alpha +2 \beta -2 \beta  \sqrt{1-k^2}-2 \alpha  k^2 z^2-2 \beta  k^2 z^2-k^2 z^2+1\right)}{\left(z^2-1\right) \left(k^2 z^2-1\right)}\right)P_n(z)=0.
\end{align}
 Let $k^2\rightarrow0^+,\; \beta\rightarrow\infty$ and $k^2\beta=t$ be fixed. The weight  $(1-x^2)^\alpha(1-k^2x^2)^\beta$  reduces to the weight $(1-x^2)^\alpha{\rm e}^{-t x^2}.$
 Assuming  $k=t^{1/4}/\sqrt {n},\;\beta=n \sqrt {t} $, equation \eqref{BC2} reduces to the following equation as $n\to \infty$:
\begin{align}\label{BC3}
P''_n(z)&+\left(\frac{\alpha +1}{z-1}+\frac{\alpha +1}{z+1}-2tz\right)P'_n(z)-\frac{n \left(n+2 \alpha +1- \left(2 z^2-1\right)t\right)}{z^2-1}P_n(z)=0.
\end{align}
Let $u=z^2$ and $\widehat{P}_n(u):=P_n(\sqrt u)$.
Substituting this into \eqref{BC3}, after some direct calculations, we obtain the confluent Heun equation \eqref{BC4}.
\end{proof}

\begin{corollary}
When $k=0$, the orthogonal polynomials reduce to the Jacobi polynomials, and the confluent Heun equation reduces to the Jacobi differential equation \begin{align}\label{BC5}
P''_n(z)&+\left(\frac{\alpha +1}{z-1}+\frac{\alpha +1}{z+1}\right)P'_n(z)-\frac{n \left(n+2 \alpha +1\right)}{z^2-1}P_n(z)=0.
\end{align}
\end{corollary}

\begin{proof}
If  $k=0$, then the weight $(1-x^2)^\alpha(1-k^2x^2)^\beta$ reduces to the classical Jacobi weight $(1-x)^\alpha(1+x)^\alpha$ 
(see also Basor and Chen \cite{BasorChen2005}). We can proceed similarly as in   Corollary \ref{Zc}. Also put $k=0$ in \eqref{BC2}  and $t=k^2\beta=0$ in \eqref{BC3}.
\end{proof}

\begin{rem}

Recall from Theorem \ref{Jacobi1} that when $k^2\rightarrow0^+, \;\beta\rightarrow\infty$ and $k^2\beta=t$ is fixed, the weight  $(1-x^2)^\alpha(1-k^2x^2)^\beta$ reduces to  $(1-x^2)^\alpha{\rm e}^{-t x^2}$. Then the normalization constant for orthogonal polynomials reads
\begin{equation*}
  h_j(t)=\int_{-1}^1P_j^2(x,t)(1-x^2)^\alpha{\rm e}^{-t x^2}{\rm d}x.
\end{equation*}
Changing  the variable $x^2=s$, we see that  the even case and the odd case of $j$ have different normalization, in particular,
 \begin{align*}
  h_{2n}(t)&=\int_{-1}^1P_{2n}^2(x,t)(1-x^2)^\alpha{\rm e}^{-t x^2}{\rm d}x\\
           &=2\int_{0}^1P_{2n}^2(\sqrt{s},t)(1-s)^\alpha\frac{{\rm e}^{-t s}}{2\sqrt{s}}{\rm d}s\\
           &=\int_{0}^1\widetilde{P}_n^2(s,t)s^{-\frac{1}{2}}(1-s)^\alpha{\rm e}^{-t s}{\rm d}s=:\widetilde{h}_n(t),
\end{align*}
and
\begin{align*}
  h_{2n+1}(t)&=\int_{-1}^1P_{2n+1}^2(x,t)(1-x^2)^\alpha{\rm e}^{-t x^2}{\rm d}x\\
           &=2\int_{0}^1P_{2n+1}^2(\sqrt{s},t)(1-s)^\alpha\frac{{\rm e}^{-t s}}{2\sqrt{s}}{\rm d}s\\
           &=\int_{0}^1\widehat{P}_n^2(s,t)s^{\frac{1}{2}}(1-s)^\alpha{\rm e}^{-t s}{\rm d}s=:\widehat{h}_n(t).
\end{align*}
Here
\begin{align*}
  P_{2n}(\sqrt{s},t)&=(\sqrt{s})^{2n}+p(2n,t)(\sqrt{s})^{2n-2}+\ldots+P_{2n}(0,t)\\
  &=s^n+\widetilde{p}(n,t)s^{n-1}+\ldots+\widetilde{P}_n(0,t):=\widetilde{P}_n(s,t),
\end{align*}
and
\begin{align*}
    P_{2n+1}(\sqrt{s},t)&=(\sqrt{s})^{2n+1}+p(2n,t)(\sqrt{s})^{2n-1}+\ldots+const\cdot \sqrt{s}\\
  &=\sqrt{s}\left(s^n+\widehat{p}(n,t)s^{n-1}+\ldots+const\right):=\sqrt{s} \widehat{P}_n(s,t).
\end{align*}
The polynomials $\widetilde{P}_n(s,t)$ and $\widehat{P}_n(s,t)$ are monic polynomials of degree $n$ in the variable $s$ and they are orthogonal with respect to $s^{-\frac{1}{2}}(1-s)^\alpha{\rm e}^{-t s}$ and $s^{\frac{1}{2}}(1-s)^\alpha{\rm e}^{-t s}$ over $(0,1]$ respectively. These weights are the special cases of the weight $s^{\alpha}(1-s)^\beta{\rm e}^{-t s}$ for $\beta=-\frac{1}{2}$ and $\beta=\frac{1}{2}$ respectively, see\cite{ZGCZ}.

The Hankel determinants generated by $s^{-\frac{1}{2}}(1-s)^\alpha{\rm e}^{-t s}$ and $s^{\frac{1}{2}}(1-s)^\alpha{\rm e}^{-t s}$, $0< s\leq 1$ are defined by
\begin{align*}
  \widetilde{D}_m(t):= &{\rm det}\left(\int_{0}^\infty s^{i+j-\frac{1}{2}}(1-s)^\alpha{\rm e}^{-t s}{\rm d}s \right)_{i,j=0}^{n-1}=\prod_{l=0}^{m-1}\widetilde{h}_l(s),\\
 \widehat{ D}_m(t):= &{\rm det}\left(\int_{0}^\infty s^{i+j+\frac{1}{2}}(1-s)^\alpha{\rm e}^{-t s}{\rm d}s\right)_{i,j=0}^{n-1}=\prod_{l=0}^{m-1}\widehat{h}_l(s)
\end{align*}
respectively. Hence,
\begin{equation*}
 D_n(t)=\prod_{j=0}^{n-1}h_j(t)=\begin{cases}
\widetilde{D}_{k+1}\widehat{ D}_k&~~ \text{n=2k+1,}\\
\widetilde{D}_k\widehat{ D}_k& ~~\text{n=2k.}
\end{cases}
\end{equation*}
The Hankel determinants for large $n$ are very interesting, we shall not pursue this subject further as it lies beyond the scope of this paper. See   Lyu, Chen and Fan \cite{LC2} for the Gaussian weight, also the monograph by Szeg$\ddot{o}$ \cite{Szego}.
\end{rem}

\section{Laguerre type weights}

In this section we consider two deformed Laguerre type weights: $x^\alpha(x+t)^\lambda{\rm e}^{-x}$ (see \cite{EC,CM2012}) and $x^\alpha{\rm e}^{-x-t/x}$ (see \cite{ChenIts2010}). The weight $x^\alpha(x+t)^\lambda{\rm e}^{-x}$ appear in multiple-input multiple-output (MIMO) wireless communication systems. The technique of ladder operators was used to study Hankel determinants  and to show connection to the  Jimbo-Miwa-Okamoto $\sigma$-form of the fifth  Painlev\'{e} equation in  \cite{EC,CM2012}. The deformed Laguerre weight $x^\alpha{\rm e}^{-x-t/x}$ (see \cite{ChenIts2010}) appear in mathematical physics and integrable quantum field theory of finite temperature (see \cite{L2001}). The technique of ladder operators and the  Riemann-Hilbert approach gives  connection of recurrence coefficients and of the  logarithmic derivative of the Hankel determinant to the solutions of the third   Painlev\'{e} equation  (or the $\sigma$-form of it).

\subsection{$x^\alpha(x+t)^\lambda{\rm e}^{-x}, \;x\in[0,\infty),\; t,\,\alpha,\,\lambda>0$}

From \cite{EC}, the second order differential equation satisfied by monic polynomials $P_n(x)$ orthogonal  with respect to the weight $w(x,t,\lambda)=x^\alpha(x+t)^\lambda{\rm e}^{-x}$ is of the following form:
\begin{align}\label{PEC1}
P_n''(z)+Q_n(z,t)P_n'(z) +S_n(z,t)P_n(z)=0,
\end{align}
where
\begin{align*}
Q_n(z,t)=&\frac{\alpha+1}{z}+\frac{\lambda+1}{z+t}-1-\frac{1}{z+t[1-R_n(t)]},\\
S_n(z,t)=&\frac{t \left(r_n(t)+n R_n(t)\right)}{z (t+z) \left(z+t \left(1-R_n(t)\right)\right)}+\frac{n-\sum_{j=0}^{n-1}R_j(t)}{z}+\frac{\sum_{j=0}^{n-1}R_j(t)}{t+z}
\end{align*}
with $R_n(t)$, $r_n(t)$ defined by
\begin{align}
 r_n(t)&=\frac{t  R'_n(t)+\lambda -R_n(t) \left(t+2n+\alpha +\lambda -t R_n(t)\right)}{2}\label{ECr}, \\
 \sum_{j=0}^{n-1}R_j(t)&=\frac{n (\alpha +\lambda +n)-t r_n(t)-\beta _n(t)}{t}\label{ECSR}
\end{align}
and
\begin{equation}
\beta _n(t)=\frac{1}{1-R_n(t)}\left(n (\alpha +n)+\frac{r_n(t){}^2-\lambda  r_n(t)}{R_n(t)}+(\alpha +\lambda +2 n) r_n(t)\right).\label{ECb}
\end{equation}

Obviously the coefficients of \eqref{PEC1} are given in terms of $n,\;\lambda,\;t$, $R_n(t)$ and its derivative. First we will obtain the asymptotic formula for $R_n(t)$ and then show how the second order differential equation  \eqref{PEC1} can be reduced to the Heun equation.
\begin{proposition}
For $t>0$ and $n\rightarrow\infty$ we have
\begin{align}\label{}
 R_n(t)= \frac{\lambda }{2(nt)^{1/2}}-&\frac{  4 \lambda\left(\alpha ^2+t^2+2 t (\alpha +\lambda +1)\right)-\lambda}{64  (nt)^{3/2}}\notag
 \\&+\frac{\lambda ^2 \left(4 t^2-4 \alpha ^2+1\right)}{64t^2 n^2 }+\mathcal{O}\left(\frac{1}{n^{5/2}}\right).\label{ECR}
\end{align}
\end{proposition}
\begin{proof}
From \cite{EC} the auxiliary quantity $R_n(t)$ is given by
$$R_n(t):=\frac{\lambda}{h_n}\int_0^\infty\frac{P_n(x)^2}{x+t}x^\alpha(x+t)^\lambda{\rm e}^{-x} {\rm d}x,\quad t,\alpha,\lambda>0.$$ It satisfies
\begin{align}\label{R2}
R_n''(t)&= \frac{1}{2 t^2 \left(R_n(t)-1\right) R_n(t)}\Big\{\lambda ^2+4 t  (\alpha +\lambda +2 n+t+1)R_n(t)^3\notag\\
&-t^2 R_n'(t)^2+2 t^2 R_n(t)^5-t  (2 \alpha +2\lambda +2+4 n+5 t)R_n(t)^4
\notag\\
&- \left(2 t R_n'(t)+\alpha ^2-\lambda ^2+2 t (\alpha +\lambda +2 n+1)+t^2\right)R_n(t)^2
\notag\\
&+2  \left(t^2 R_n'(t)^2+t R_n'(t)-\lambda ^2\right)R_n(t)\Big\}.
\end{align}
Disregarding  the derivative part of this nonlinear second order differential equation, for small $\alpha$ we have
$$(\widetilde{R}_n(t)-1)^2 \left(\lambda ^2-t  (2 \lambda +4 n+t+2)\widetilde{R}_n(t)^2+2 t^2 \widetilde{R}_n(t)^3\right)=0.$$
The solutions are given by
\begin{align*}
  \widetilde{R}_{n1,2}(t)=&1,\\
    \widetilde{R}_{n3,4}(t)=&\pm\frac{\lambda }{2 \sqrt{nt}}\mp\frac{\lambda   (2 \lambda +t+2)}{16 \sqrt{t}n^{-3/2}}+\frac{\lambda ^2}{16 n^2}\pm\frac{3  \lambda  (2 \lambda +t+2)^2}{256 \sqrt{t}n^{-5/2}}+\mathcal{O}\left(n^{-3}\right),\\
  \widetilde{R}_{n5}(t)=&\frac{2 n}{t}+\frac{2 \lambda +t+2}{2 t}-\frac{\lambda ^2}{8 n^2}+\frac{\lambda ^2 (2 \lambda +t+2)}{16 n^3}+\mathcal{O}\left(n^{-7/2}\right).
\end{align*}
Assuming that $R_n(t)$ has the form
$$R_n(t)=\sum^{\infty}_{j=1}b_j(t)n^{-j/2},\quad n\rightarrow\infty,$$
and substituting  into \eqref{R2}, we obtain \eqref{ECR}.
\end{proof}

\begin{theorem}
As $n\rightarrow\infty$, the monic polynomials $P_n(x)$ orthogonal with respect to $x^\alpha(x+t)^\lambda{\rm e}^{-x}$ over $[0,\infty)$  satisfy the confluent Heun equation
\begin{align}\label{PEC2}
 \widehat{P}_n''(u)+\left(\frac{\widetilde{\gamma}}{u}+\frac{\widetilde{\delta}}{u-1}+\widetilde{\epsilon}\right)\widehat{P}_n'(u)+\frac{\widetilde{\alpha} u-\widetilde{q}}{u(u-1)}\widehat{P}_n(u)=0
\end{align}
with parameters
$$\widetilde{\gamma}=\alpha+1,\quad\widetilde{ \delta}=2\eta+\lambda,\quad \widetilde{\epsilon}=t,\quad \widetilde{\alpha}=t(\eta-n),$$$$ \widetilde{q}=-\frac{(2 \eta +\lambda ) (2 (\alpha +\eta )+\lambda )-4 \lambda  \sqrt{nt}+4 n t+2 \lambda  t}{4} .$$
Here  $\widehat{P}_n(u):=(u-1)^{-\eta} P_n(-tu).$
\end{theorem}

\begin{proof}
Substituting \eqref{ECr}--\eqref{ECb} into \eqref{PEC1}, the coefficients of \eqref{PEC1} are given in terms of $R_n(t)$ and $R'_n(t)$. Substituting \eqref{ECR} and taking $n\rightarrow\infty$ we obtain
\begin{align}\label{PEC3}
P_n''(z)&+\Big(\frac{\alpha+1}{z}+\frac{\lambda}{z+t}-1\Big)P_n'(z) \notag\\
&+\bigg(\frac{n}{z}-\frac{\lambda  (4\sqrt{n t}-2\alpha-2t-\lambda)}{4z (t+z)}+\frac{\lambda   t}{2 z (t+z)^2}\bigg)P_n(z)=0.
\end{align}
Let $$z=-t u.$$
Then $\widehat{P}_n(u):=(u-1)^{-\eta} P_n(-tu)$,
where $$\eta=\frac{1-\lambda\pm\sqrt{\lambda ^2+1}}{2}, $$
satisfies the confluent Heun equation \eqref{PEC2}.
\end{proof}


\begin{corollary}\label{ECco}
When $\lambda=0$, the weight $x^\alpha(x+t)^\lambda{\rm e}^{-x}$ reduces to the classical Laguerre weight $x^{\alpha}{\rm e}^{-x}$, the polynomials $P_n(z)$ reduce to the Laguerre polynomials $L_n^\alpha(z)$ and equation \eqref{PEC3} reduces to the Laguerre equation
\begin{align}\label{PEC5}
P_n''(z)&+\Big(\frac{\alpha+1}{z}-1\Big)P_n'(z) +\frac{n}{z}P_n(z)=0.
\end{align}
\end{corollary}

\begin{proof} Let us use first ladder operators approach to derive  \eqref{PEC5}.
For the weight function $x^\alpha {\rm e}^{-x}$ we have $v(x)=-\alpha \ln x+x$ and $v'(x)=-\frac{\alpha}{x}+1$.
From \eqref{A}--\eqref{B} we have
\begin{align*}
  A_n(z)&=\frac{1}{h_n}\int_0^\infty\;\frac{\alpha}{yz}P^2_n(y)y^\alpha{\rm e}^{-y}\;{\rm d}y=\frac{1}{z},\\
   B_n(z)&=\frac{1}{h_n}\int_0^\infty\;\frac{\alpha}{yz}P_n(y)P_{n-1}(y)y^\alpha{\rm e}^{-y}\;{\rm d}y=-\frac{n}{z}.
\end{align*}
Recalling \eqref{H}, we obtain
\begin{align*}
 -( v'(z)+\frac{A'_n(z)}{A_n(z)}) &=\frac{\alpha+1}{z}-1,  \\
  B'_n(z)-B_n(z)\frac{A'_n(z)}{A_n(z)}+\sum_{j=0}^{n-1}A_j(z)&=\frac{n}{z},
\end{align*}
which gives \eqref{PEC5}.
\end{proof}

From the expressions for $r_n(t),\;\beta_n(t),\;\sum_{j=0}^{n-1}R_j(t)$ in terms of $R_n(t)$, it is easy to obtain their asymptotic expressions.

\begin{rem}
The auxiliary quantities $r_n(t)$, $\beta _n(t)$ and $\sum_{j=0}^{n-1}R_j(t)$ have  the following asymptotic expressions when $n$ is large:

\begin{align}
 r_n(t)&=\frac{\lambda }{2}-\frac{\lambda n^{1/2}}{2 t^{1/2}}-\frac{ \lambda \left[4 t (3t+2\alpha +2\lambda )-4 \alpha ^2+1\right]}{64 t^{3/2}n^{1/2}}\notag\\
 &~~~~~~~~~~~~~~~~~~~~~~~~~~~~~~~~~~+\frac{\lambda ^2 \left(4 \alpha ^2+4 t^2-1\right)}{64 n t^2}+\mathcal{O}\left(\frac{1}{n^{3/2}}\right)\label{ECr1},\\
 \beta _n(t)&=n^2+n (\alpha +\lambda )-\frac{\lambda  (t n)^{1/2}}{2} +\frac{1}{4} \lambda  (2 \alpha +\lambda )\notag\\
 &+\frac{\lambda   \left(4 t (t-2\alpha -2\lambda )-12 \alpha ^2+3\right)}{64 \sqrt{n t}}+\frac{\left(1-4 \alpha ^2\right) \lambda ^2}{32 n t}+\mathcal{O}\left(\frac{1}{n^{3/2}}\right)\label{ECb1}
,\\
  \sum_{j=0}^{n-1}R_j(t)&=\frac{\lambda  n^{1/2}}{t^{1/2}}-\frac{\lambda  (\lambda +2 (\alpha +t))}{4 t}+\frac{\lambda  (4 t (4 \alpha +2 \lambda +3 t)+1)}{64 t^{3/2}n^{1/2}}\notag\\
  &~~~~~~~~~~~~~~~~~~~~~~~~~~~~~~~~~~-\frac{\lambda ^2 \left(1-4 \alpha ^2+4 t^2\right)}{64 n t^2}+\mathcal{O}\left(\frac{1}{n^{3/2}}\right).\label{ECSR1}
\end{align}
\end{rem}

\begin{rem}
From the relation $\alpha_n(t)=2n+1+\alpha+\lambda-tR_n(t)$ (see\cite{EC}) and the asymptotic expression \eqref{ECR} we obtain
\begin{equation}\label{ECa1}
 \alpha_n(t)=2n+1+\alpha+\lambda-\frac{\lambda \sqrt{t} }{2 \sqrt{ n}}+\frac{\lambda  \left[4 \left(\alpha ^2+t^2+2 t (\alpha +\lambda +1)\right)-1\right]}{64\sqrt{t} n^{3/2}}+\mathrm{O}\left(\frac{1}{n^2}\right).
\end{equation}
\end{rem}

Next  we will use Dyson's Coulomb fluid method to   check the correctness of this result.

\begin{proposition}
From Dyson's Coulomb fluid approximation theory  we obtain
\begin{align}
  &\alpha_n(t)  \sim2n+\alpha+\lambda-\frac{\lambda  \sqrt{t}}{2 \sqrt n}+\frac{\lambda   \left((\alpha +t)^2+2 \lambda  t\right)}{16 \sqrt{t}n^{3/2}} \notag\\
  &~~~~~~~~~~~~~~~~~~~~~~~~~~~~~~~~~~~~~~~+\frac{\lambda ^2 \left(\alpha ^2-t^2\right)}{16n^2 t}+\mathcal{O}\left(\frac{1}{n^{5/2}}\right) , \quad n\rightarrow\infty,\label{ECa2}\\
 & \beta_n(t)  \sim n^2+n (\alpha +\lambda )-\frac{ \lambda(nt)^{1/2}}{2}  +\frac{\lambda  (2 \alpha +\lambda )}{4} +\mathcal{O}\left(\frac{1}{n^{1/2}}\right), \quad n\rightarrow\infty.\label{ECb2}
\end{align}
\end{proposition}

\begin{proof}
For weight function $w(x)=x^\alpha(x+t)^\lambda{\rm e}^{-x}$  we have
$$v(x)=-\alpha \ln x-\lambda \ln(x+t)+x\quad \text{and} \quad v'(x)=-\frac{\alpha}{x}-\frac{\lambda}{x+t}+1.$$
Substituting $v'(x)$ into \eqref{supp2}, combining with formulas \eqref{int1}--\eqref{int3} and \eqref{int5},  we obtain following algebraic equations:
\begin{align}\label{}
  \frac{\alpha}{\sqrt{ab}}+\frac{\lambda}{\sqrt{(a+t)(b+t)}}-1=0,\label{ECag1}\\
  2n+\alpha+\lambda-\frac{\lambda t}{\sqrt{(a+t)(b+t)}}-\frac{a+b}{2}=0\label{Ecag2}.
\end{align}

Let $Y(t):=(a+b)/2$ and $Y(0)=2n+\alpha+\lambda$. Solving for $1/\sqrt{ab}$ from the sum of \eqref{Ecag2} and \eqref{ECag1} multiplied by $t$, we get
\begin{equation}\label{ECsq}
  \frac{1}{\sqrt{ab}}=\frac{Y(t)-(2n+\alpha+\lambda-t)}{t\alpha}.
\end{equation}
Substituting $1/\sqrt{ab}$ into \eqref{ECag1} and using the expression for  $Y(t)$, we  obtain
\begin{equation*}
  \frac{Y(t)-(2n+\alpha+\lambda)}{t}+\lambda/\sqrt{t^2+2tY(t)+\left(\frac{t\alpha }{Y(t)-(2n+\alpha+\lambda-t)}\right)^2}=0.
\end{equation*}
Eliminating the square root we get
\begin{align}
  (Y(t)-\alpha -\lambda -2 n)^2 &\left[\left(t^2+2 t Y(t)\right) (Y(t)-\alpha -\lambda -2 n+t)^2+\alpha ^2 t^2\right]\notag\\
  &=\lambda ^2 t^2 (Y(t)-\alpha -\lambda -2 n+t)^2.\label{ECY}
\end{align}
For small $\alpha$ and $n\rightarrow\infty $ we consider equation
\begin{equation*}
  t (-\lambda -2 n+t+\widetilde{Y}(t))^2 \left((t+2\widetilde{ Y}(t)) (\lambda +2 n-\widetilde{Y}(t))^2-\lambda ^2 t\right)=0,
\end{equation*}
which is solved by
\begin{align*}
  \widetilde{Y}_{1,2}(t) &=2n+\lambda-t, \\
  \widetilde{Y}_{3,4}(t) &=2 n+\lambda+\frac{\lambda}{2}  \left( \frac{t}{n}\right)^{1/2}  -\frac{\sqrt{t} (2 \lambda +t)\lambda }{16}  \left(\frac{1}{n}\right)^{3/2} -\frac{\lambda ^2 t}{16 n^2}+\mathcal{O}\left(\frac{1}{n^{5/2}}\right),\\
 \widetilde{Y}_{5}(t) &= -\frac{t}{2}+\frac{\lambda ^2 t}{8 n^2}+\mathcal{O}\left(\frac{1}{n^{5/2}}\right).
\end{align*}
Assuming that $Y(t)$ has the form
\begin{equation*}
  Y(t)=\sum _{j=0}^{\infty} a_j(t) n^{1-j/2}, \quad n\rightarrow\infty,
\end{equation*}
and substituting  the expression above into \eqref{ECY}, we obtain $Y(t)$, which is the same as $\alpha_n(t)$ in \eqref{ECa2}. Then combining with \eqref{ECsq} we obtain
\begin{equation*}
\beta_n(t)\sim \frac{\left((a+b)/2\right)^2-ab}{4}=\frac{1}{4} \left(Y(t)^2-\left(\frac{\alpha  t}{Y(t)-(\alpha +\lambda +2 n-t)}\right)^2\right),
\end{equation*}
which gives \eqref{ECb2}.

 Comparing \eqref{ECa1}, \eqref{ECb1} with \eqref{ECa2}, \eqref{ECb2} respectively we see that the ladder operator approach  and the Coulomb fluid approximation  method yield the same asymptotic expressions for $\alpha_n(t)$, $\beta_n(t)$ when $n$ is large.

\end{proof}

Next we consider equation \eqref{PEC1} with coefficients depending on $R_n(t)$ and its deivative, without any reference to orthogonal polynomials.

\begin{proposition}
If $R_n(t)$ satisfies the Riccati equation
$$tR_n'(t)=t R_n(t)^2 -(\alpha +\lambda +t) R_n(t)+\lambda $$
with the solution
\begin{equation}\label{BCd}
R_n(t)= \frac{\alpha  C_2 U(\alpha +1,\alpha +\lambda +1,t)+L_{-\alpha -1}^{\alpha +\lambda }(t)}{C_2 U(\alpha ,\alpha +\lambda ,t)+L_{-\alpha }^{\alpha +\lambda -1}(t)}+1,
\end{equation}
then $\widetilde{P}_n(u):=P_n(-t u)$ satisfy equation for the derivative  of the confluent Heun function \eqref{dC}
 with parameters
$$\widetilde{\gamma} = \alpha ,\;\widetilde{\delta} = \lambda ,\; \widetilde{\epsilon} = t,\; \widetilde{q}= -(n+1)t (1-R_n(t)),\; \widetilde{\alpha}=-(n+1) t.$$
\begin{proof}
Let $z=-tu$, then $\widetilde{P}_n(u):=P_n(-tu)$, which gives
\begin{align*}
\widetilde{P}_n''(u)&+\Big(\frac{\alpha+1}{u}+\frac{\lambda+1}{u-1}+t-\frac{1}{u-[1-R_n(t)]}\Big)\widetilde{P}_n'(u) \\ &-\left(\frac{r_n(t)+R_n(t)}{u(u-1)[u-(1-R_n(t))]} +\frac{nt(u-1)+\sum_{j=0}^{n-1}tR_j(t)}{u(u-1)}\right)\widetilde{P}_n(u)=0.
\end{align*}
This equation is the equation for the   derivative of the confluent Heun function  when
$$tR_n'(t)= tR_n(t)^2-(\alpha +\lambda +t) R_n(t)+\lambda .$$
Solving this differential equation we obtain \eqref{BCd} and in the case when the constant $C_2=0$ we obtain
$$R_n(t)=\frac{\lambda M(\alpha;\alpha+\lambda+1;t)}{(\alpha+\lambda)M(\alpha;\alpha+\lambda;t)}.$$

Note that  if  $R_n(t)$  satisfies both the Riccati equation and  \eqref{R2}, then $n=-1$.
\end{proof}
\end{proposition}

\subsection{$x^\alpha {\rm e}^{-x-t/x}, \;x\in(0,\infty),\,\alpha,\,t>0$}

The weight  $w(x,\alpha,t)=x^\alpha {\rm e}^{-x-t/x}$ was studied by Chen and Its  \cite{ChenIts2010} for finite $n$, Chen and Chen et al \cite{ChenChenJMP2015}  for  $n\to\infty$.

The second order linear differential equation satisfied by $P_n(z)$ (see Chen and Its \cite{ChenIts2010}) is given by
\begin{align}\label{CI1}
P_n''(z)+Q_n(z,t)P_n'(z) +S_n(z,t)P_n(z)=0,
\end{align}
where
\begin{align*}
Q_n(z,t)=\frac{\alpha+2}{z}+\frac{t}{z^2}-\frac{1}{z+R_n(t)}-1,\\
S_n(z,t)=\sum _{j=0}^{n-1} \frac{R_j(t)}{z^2}+\frac{n \left((z-1) R_n(t)+z^2\right)-r_n(t)}{z^2 \left(R_n(t)+z\right)}.
\end{align*}
Auxiliary quantities $R_n(t)$, $r_n(t)$ satisfy
\begin{align}
  r_n(t)&=\frac{t+tR_n'(t)-(2n+1+\alpha+R_n(t))R_n(t)}{2}, \label{CIr}\\
  \sum _{j=0}^{n-1} R_j(t)&=-n(n+\alpha)-r_n(t)+\beta_n(t),\label{CISR}
\end{align}
where
\begin{equation}\label{CIb}
  \beta_n(t)=\frac{1}{R_n(t)}\left[nt-(2n+\alpha)r_n(t)-\frac{r_n^2(t)-tr_n(t)}{R_n(t)}\right].
\end{equation}

To obtain the asymptotic expression of $R_n(t)$, the method of double scaling will be used. Let  $n\rightarrow\infty,\;t\rightarrow0^+ $  and let   $s=(2n+\alpha+1)t$ be  fixed. It should be pointed out that the asymptotic expression of $\mathbb{R}_n(s)$ was given in \cite{ChenChenJMP2015}, see the following proposition. For convenience of the reader we use the hollow symbol to define a new function of   $s$, that is  $\mathbb{R}_n(s)=R_n(s/(2n+\alpha+1))$.

\begin{proposition}\label{CIp}\cite{ChenChenJMP2015}
Let  $n\rightarrow\infty$, $t\rightarrow0^+$ and $s=(2n+\alpha+1)t$ be fixed.

Case I: for large $s$ we have
\begin{align}
\mathbb{R}_n(s)&= \frac{1 }{2n+\alpha+1}\left(s^{\frac{2}{3}}-\frac{\alpha}{3}s^{\frac{1}{3}}+\frac{\alpha(\alpha^2-1)}{81}s^{-\frac{1}{3}}+\frac{\alpha^2(\alpha^2-1)}{243}s^{-\frac{2}{3}}
\right)+\mathcal{O}\left(s^{-2}\right).\label{CIRs1}
\end{align}

Case II: for small $s$ we have
\begin{align}
\mathbb{R}_n(s)&= \frac{1 }{2n+\alpha+1}\left(\frac{s}{\alpha }-\frac{s^2}{\alpha ^2 \left(\alpha ^2-1\right)}+\frac{3 s^3}{\alpha ^3 \left(\alpha ^2-1\right) \left(\alpha ^2-4\right)}\right)+\mathcal{O}\left(s^4\right)\label{CIRs2},
\end{align}
where $\alpha\neq\mathbb{Z}$.
\begin{proof}

Define
$$\widetilde{R}(s)=\lim_{n\rightarrow\infty}\frac{R_n(s/(2n+\alpha+1))}{s/(2n+\alpha+1)}.$$
Then $\widetilde{R}(s)$ satisfies
$$\widetilde{R}''(s)=\frac{\widetilde{R}'(s){}^2}{\widetilde{R}(s)}-\frac{\widetilde{R}'(s)}{s}+\frac{\widetilde{R}^2(s)}{s}+\frac{\alpha}{s^2}-\frac{1}{s^2\widetilde{R}(s)}$$
with the initial conditions $\widetilde{R}(0)=1/\alpha$, $\widetilde{R}'(0)=-1/(\alpha^2(\alpha^2-1))$.

For large $s$ we have
 $$\widetilde{R}(s)= s^{-\frac{1}{3}}-\frac{\alpha}{3}s^{-\frac{2}{3}}+\frac{\alpha(\alpha^2-1)}{81}s^{-\frac{4}{3}}+\frac{\alpha^2(\alpha^2-1)}{243}s^{-\frac{5}{3}}
+\mathcal{O}\left(s^{-2}\right).$$

For small $s$ we have
\begin{align*}
\widetilde{R}(s)=& \frac{1}{\alpha }-\frac{s}{\alpha ^2 \left(\alpha ^2-1\right)}+\frac{3 s^2}{\alpha ^3 \left(\alpha ^2-1\right) \left(\alpha ^2-4\right)}\\
&-\frac{6(2\alpha^2-3)s^3}{\alpha^4(\alpha^2-1)^2(\alpha^2-4)(\alpha^2-9)}+\mathcal{O}\left(s^4\right),
\end{align*}
see Chen and Chen \cite{ChenChenJMP2015} for details.

From
$$\widetilde{R}(s)=\lim_{n\rightarrow\infty}\frac{\mathbb{R}_n(s)}{s/(2n+\alpha+1)}$$
we deduce
$$\mathbb{R}_n(s)=\frac{s}{2n+\alpha+1}\widetilde{R}(s),\quad n\rightarrow \infty.$$
This gives \eqref{CIRs1} and \eqref{CIRs2}.
\end{proof}
\end{proposition}

The coefficients of \eqref{CI1}  are given in terms of $R_n(t)$. Next we show that \eqref{CI1} reduces to  the double confluent Heun equations, both for small and for large $s$.

\begin{theorem}
Let $n\rightarrow\infty$, $t\rightarrow0^+$ and $s=(2n+\alpha+1)t$ be fixed. Equation  \eqref{CI1} reduces to  the   double confluent Heun equation
\begin{align}\label{CI3}
P_n''(z)+\Big(\frac{\widetilde{\gamma}}{z^2}+\frac{\widetilde{\delta}}{z}+\widetilde{\epsilon}\Big)P_n'(z) +\left(\frac{\widetilde{\alpha} z-\widetilde{q}}{z^2}\right)P_n(z)=0
\end{align}
with the following parameters.

Case I: for large $s$
\begin{align*}
\widetilde{\gamma}=\frac{3s+3 s^{2/3}-\alpha  s^{1/3}}{6 n },\quad \widetilde{\delta}=\alpha+1, \quad \widetilde{\epsilon}=-1,\\ \widetilde{\alpha}=n,\quad \widetilde{q}=\frac{-6 \alpha ^2-27 s^{2/3}+18 \alpha  s^{1/3}+1}{36} .
\end{align*}

Case II: for small $s$
$$
 \widetilde{\gamma}=\frac{s(\alpha+1)}{2n \alpha},\quad \widetilde{\delta}=\alpha+1,\quad \widetilde{\epsilon}=-1,\quad \widetilde{\alpha}=n,\quad \widetilde{q}=-\frac{s}{2\alpha}.$$
\end{theorem}
\begin{proof}
Substituting \eqref{CIr}--\eqref{CIb} into \eqref{CI1}, the coefficients of \eqref{CI1} can be expressed in terms of $R_n(t)$ and $R'_n(t)$.  Setting $s=(2n+\alpha+1)t$ and taking  $n\to\infty $, we substitute \eqref{CIRs1}, \eqref{CIRs2} and obtain the results.
\end{proof}

The following asymptotic expansions hold.
\begin{rem}
Let  $n\rightarrow\infty$, $t\rightarrow0^+$ and $s=(2n+\alpha+1)t$ be fixed.

Case I: for large $s$ we have
\begin{align}
r_n(s)&=\frac{s}{4 n}+\frac{1}{6} \left(\frac{1}{n}-3\right) s^{2/3}+\frac{\alpha  }{6}\left(1-\frac{1}{6n}\right)s^{1/3}+\mathcal{O}\left(s^{-1/3}\right),\label{CIrl}\\
\beta_n(s)&=n^2+\alpha  n +\frac{1}{12}\left(3-\frac{1}{n}\right)s^{2/3}-\frac{\alpha}{18}\left(6-\frac{1}{n} \right) s^{1/3}\notag\\
&~~~~~~~~~~~~~~~~~~~~~~~~~~~~~~~~~~~~~~~~+\frac{6 \alpha ^2-1}{36} +\mathcal{O}\left(s^{-1/3}\right),\label{CIbl}\\
\sum _{j=0}^{n-1} \mathbb{R}_j(s)&=-\frac{s}{4 n}+\frac{1}{4}\left(3-\frac{1}{ n}\right) s^{2/3}-\frac{\alpha}{12}\left(6-\frac{1}{n}\right) s^{1/3}\notag\\&~~~~~~~~~~~~~~~~~~~~~~~~~~~~~~~~~~~~~~~~+\frac{6 \alpha ^2-1}{36} +\mathcal{O}\left(s^{-1/3}\right) .\label{CISRl}
\end{align}

Case II: for small $s$ we have

\begin{align}
r_n(s)&=-\frac{s ( 2 n-\alpha-1)}{4 \alpha  n}+\frac{(n-1) s^2}{2 \alpha ^2 \left(\alpha ^2-1\right) n}\notag\\
&~~~~~~~~~~~~~~~~~~~~~~~~~~~~~~~~~~-\frac{3 (2 n-3) s^3}{4 \alpha ^3 \left(\alpha ^2-4\right) \left(\alpha ^2-1\right) n}+\mathcal{O}\left(s^4\right),\label{CIr2}\\
\beta_n(s)&=n (n+\alpha )+\frac{(n-1) s^2}{4 \alpha ^2 \left(\alpha ^2-1\right) n}-\frac{(2 n-3) s^3}{2 \alpha ^3 \left(\alpha ^2-4\right) \left(\alpha ^2-1\right) n}+\mathcal{O}\left(s^3\right),\label{CIb2}\\
\sum _{j=0}^{n-1} \mathbb{R}_j(s)&=\frac{s (2n-\alpha -1)}{4 \alpha  n}-\frac{(n-1) s^2}{4 \alpha ^2 \left(\alpha ^2-1\right) n}\notag\\
&~~~~~~~~~~~~~~~~~~~~~~~~~~~~~~~~~~-\frac{ (2 n-3) s^3}{4 \alpha ^3 \left(\alpha ^2-4\right) \left(\alpha ^2-1\right) n}+\mathcal{O}\left(s^4\right).\label{CISR2}
\end{align}

\begin{proof}
Since \eqref{CIr}--\eqref{CIb} are expressed in terms of $R_n(t)$ and $R_n'(t)$, setting $s=(2n+\alpha+1)t$ and combining with Proposition \ref{CIp} we obtain the results as  $s$ goes to $\infty$ and $0^+$ respectively.
\end{proof}
\end{rem}

\begin{corollary}
When $t=0$ the weight $x^\alpha {\rm e}^{-x-t/x}$ reduces to the classical Laguerre weight $x^\alpha {\rm e}^{-x}$. Equation \eqref{CI3} for small $s$  reduces to the Laguerre differential equation
\begin{align}\label{CI4}
P_n''(z)+\Big(\frac{\alpha+1}{z}-1\Big)P_n'(z) +\frac{n}{z}P_n(z)=0.
\end{align}
\begin{proof}
See the proof of   Corollary \ref{ECco}.
\end{proof}
\end{corollary}

\section{Weights with a gap}

In this section we consider  weights with a gap. The weight  ${\rm e}^{-x^2}(1-\chi_{(-a,a)}(x)),\;x\in\mathbb{R},\;a>0$ was studied in  \cite{LC2}).
It was shown that the Gaussian gap probabilities may be determined as the product of the smallest distributions of the Laguerre unitary ensemble with some special parameter. 
The weight  $(1-x^2)^\alpha(1-\chi(-a,a)(x))$ was  studied in \cite{MC2018}.
It was shown that auxiliary quantities satisfy certain second order differential equations.
 Also the connection to the  Jimbo-Miwa-Okamoto $\sigma$-form of the fifth Painlev\'e equation was obtained.
For the weight $w(x)=(A+B\theta(x-t))x^\alpha{\rm e}^{-x}$ (see \cite{LC1}) the largest eigenvalue distribution with finite $n$ and large $n$ was studied. Moreover,  the asymptotic solution after soft edge scaling was derived and the   second order differential equations for auxiliary quantities related to recurrence coefficients were obtained.The connection to the second Painlev\'{e} equation, the $\sigma$-form and a particular case of Chazy's equation was also shown.

\subsection{${\rm e}^{-x^2}(1-\chi_{(-a,a)}(x)),\;x\in\mathbb{R},\;a>0$}

In \cite{CaoC2014,LC2}, the second order differential equation for  polynomials $P_n(x)$ orthogonal with respect to the weight ${\rm e}^{-x^2}(1-\chi_{(-a,a)}(x)),\;x\in\mathbb{R},\;a>0,$ was obtained. It is of the following form:
\begin{align}\label{1LC1}
P_n''(z)+ Q_n(z,a)P_n'(z)+S_n(z,a)P_n(z)=0,
\end{align}
where
\begin{align*}
  Q_n(z,a) &=\frac{2 aR_n(a) z}{(z^2-a^2)[2(z^2-a^2)+aR_n(a)]}-2z, \\
  S_n(z,a) &=-\frac{r_n(a) \left(2 \left(a^2+z^2\right)-a R_n(a)\right)}{(a-z) (a+z) \left(2 a^2-a R_n(a)-2 z^2\right)}+2n+\frac{a\sum_{j=0}^{n-1}R_j(a)}{z^2-a^2}
\end{align*}
with
$\sum_{j=0}^{n-1}R_j(a)$ and $r_n(a)$ satisfying
\begin{align}
r_n(a)&=\frac{[R'_n(a)-R_n(a)^2+2aR_n(a)]a}{4a-2R_n(a)},\label{1LCr}\\
\sum_{j=0}^{n-1}R_j(a)&=\frac{2r_n(a)^2}{R_n(a)}+(n+r_n(a))R_n(a)-(2a+\frac{r_n(a)}{a})r_n(a)\label{1LCSR}.
\end{align}

As usually, we will use the asymptotic expressions for auxiliary quantities to  reduce the  second order differential equation to a simpler form, which turns out to be one of the Heun equations. We  consider the case $a>0$ not tending to $0$, so that the parameter $a$ appears in  the denominator of the asymptotic expansions.
\begin{theorem}
Let  $n\rightarrow\infty$ and $a>0$. The  polynomials $P_n(z)$ orthogonal with respect to the weight ${\rm e}^{-x^2}(1-\chi_{(-a,a)}(x)),\;x\in\mathbb{R},\;a>0$, satisfy the confluent Heun equation
\begin{align}\label{1LC2}
\widehat{P}_n''(u)+\Big(\frac{1} {u-1}-\frac{1} {2u}-t\Big)\widehat{P}_n'(u)+\frac{2ntu+\sqrt{2nt}}{4u(u-1)}\widehat{P}_n(u)=0.
\end{align}
Here $z=a\sqrt{u},\; t=a^2,$  and $\widehat{P}_n(u):=P_n(a\sqrt{u})$.
\end{theorem}
\begin{proof}
As shown in \cite{LC2}, the auxiliary quantity  $R_n(a)$ satisfies the following second order differential  equation
\begin{align}\label{1PL3}
R''_n=\frac{R_n-a}{R_n-2a}\frac{(R'_n)^2}{R_n}-\frac{R_nR'_n}{a(R_n-2a)}+\frac{R_n}{a}(R_n-2a)(aR_n-a^2+2n+1).
\end{align}
 Disregarding the derivative parts of the equation above, we obtain
$$\widetilde{R}_n(a)^2 (\widetilde{R}_n(a)-2 a)^2 \left(-a^2+a \widetilde{R}_n(a)+2 n+1\right)=0,$$
which is solved by
\begin{align*}
  \widetilde{R}_{n1,2}(a) =2a, \quad  \widetilde{R}_{n3}(a) =\frac{a^2-2 n-1}{a}.
\end{align*}
Assuming that $R_n(a)$ has the form
\begin{equation*}
R_n(a)=\sum_{j=0}^{\infty}b_j(a) n^{-j/2},\quad n\rightarrow \infty,
\end{equation*}
 substituting the series above into \eqref{1PL3}, with $R_n(a)\geq0$ and sending $n\rightarrow\infty$, we obtain
\begin{equation}\label{1LCR1}
R_n(a)=2 a+\frac{1}{\sqrt{2n}}-\frac{4 a^4+4 a^2-1}{16a^2 \sqrt{2}  n^{3/2}}-\frac{4 a^4+1}{32 a^3 n^2}+\mathcal{O}\left(\frac{1}{n^{5/2}}\right).
\end{equation}
Plugging \eqref{1LCr}, \eqref{1LCSR} into \eqref{1LC1}, sending $n\rightarrow\infty$ and combining with \eqref{1LCR1} we obtain
\begin{align}\label{1LC3}
P_n''(z)+\Big(\frac{2a^2} {z(z^2-a^2)}-2z\Big)P_n'(z)+\frac{2nz^2+\sqrt{2n}a}{z^2-a^2}P_n(z)=0,
\end{align}
which can be reduced to  a confluent Heun equation.

Let$$z=a\sqrt{u},\quad t=a^2.$$
Then $\widehat{P}_n(u):=P_n(a\sqrt{u})$ satisfies the confluent Heun equation \eqref{1LC2} with parameters
$$\widetilde{\gamma}=-1/2,\quad \widetilde{\delta}=1,\quad \widetilde{\epsilon}=-t,\quad \widetilde{a}=nt/2,\quad q=-\sqrt{2nt}/4.$$
\end{proof}

\begin{corollary}
The gap disappear when $a=0$. The weight $w(x)$ reduces to  the classical Gaussian weight ${\rm e}^{-x^2}$ for $x\in \mathbb{R}$. The orthogonal polynomials $P_n(z)$ reduce to the   Hermite polynomials $H_n(z)$ and \eqref{1LC3} reduces to the Hermite differential Equation\footnote{http://mathworld.wolfram.com/HermiteDifferentialEquation.html},
\begin{align}\label{1LC4}
P_n''(z)-2zP_n'(z)+2nP_n(z)=0.
\end{align}

\begin{proof}
Let us use ladder operators. For the weight   $w(x)= {\rm e}^{-x^2}$ we have $v(x)=x^2$ and $v'(x)=2x$.
 From \eqref{A}--\eqref{B} we have
\begin{align*}
  A_n(z)&=\frac{1}{h_n}\int_0^\infty\;2P^2_n(y){\rm e}^{-y^2}\;{\rm d}y=2,\\
   B_n(z)&=\frac{1}{h_n}\int_0^\infty\;2P_n(y)P_{n-1}(y){\rm e}^{-y^2}\;{\rm d}y=0.
\end{align*}
Recalling \eqref{H}, we obtain
\begin{align*}
 -( v'(z)+\frac{A'_n(z)}{A_n(z)}) &=-2z,  \\
  B'_n(z)-B_n(z)\frac{A'_n(z)}{A_n(z)}+\sum_{j=0}^{n-1}A_j(z)&=2n,
\end{align*}
which produces \eqref{1LC4}.
\end{proof}

\end{corollary}

\begin{rem}
The auxiliary quantities $\sum_{j=0}^{n-1}R_j(a)$ and $r_n(a)$ have the following expansions when $n$ is large  and $a>0$:
\begin{align*}
  r_n(a)=&-\sqrt{2n} a+a^2+\frac{1-4 a^4 }{8 \sqrt{2n} a} +\mathcal{O}\left(\frac{1}{n}\right), \\
  \sum_{j=0}^{n-1}R_j(a)=&2 a n+\frac{1}{4 a}-\frac{1}{2 \sqrt{2n}}+\frac{4 a^4+1}{32 a^3 n}+\mathcal{O}\left(\frac{1}{n^{3/2}}\right).
\end{align*}
\end{rem}

\subsection{$(1-x^2)^\alpha(1-\chi_{(-a,a)}(x)), \;x\in[-1,1],\; a\in(0,1),\, \alpha>0$}

 From the definition of $R_n(a)$ and $r_n(a)$ (see Min and Chen \cite{MC2018})    we have
\begin{equation}\label{MC1}
P_n''(z)+Q_n(z,a)P_n'(z) +S_n(z,a)P_n(z)=0,
\end{equation}
where
\begin{align*}
Q_n(z,a)=&\frac{2(\alpha+1)z}{z^2-1}+\frac{2z}{z^2-a^2}-\frac{2 z (2 \alpha +2 n+1)}{\left(a-a^3\right) R_n(a)-\left(a^2-z^2\right) (2 \alpha +2 n+1)}, \\
S_n(z,a)=&\frac{\left(1-a^2\right) r_n(a) \left(\left(a^2+z^2\right) (2 \alpha +2 n+1)+a \left(a^2-1\right) R_n(a)\right)}{\left(z^2-1\right) \left(a^2-z^2\right) \left(\left(a^2-z^2\right) (2 \alpha +2 n+1)+a \left(a^2-1\right) R_n(a)\right)}\\
&-\frac{n \left(\left(a^2-z^2\right)^2 (2 \alpha +2 n+1)+a \left(a^2-1\right) R_n(a) \left(a^2-3 z^2\right)\right)}{\left(z^2-1\right) \left(z^2-a^2\right) \left(\left(a-a^3\right) R_n(a)+\left(z^2-a^2\right) (2 \alpha +2 n+1)\right)}\\
&+\frac{(n^2+2\alpha n)(a^2-z^2)-a(1-a^2)\sum_{j=0}^{n-1}R_j(a)}{(z^2-a^2)(z^2-1)}
\end{align*}
with  $r_n(a)$, $\sum_{j=0}^{n-1}R_j(a)$, $\beta_n(a)$ given by
\begin{align}
 r_n(a)&=\frac{a \left(-\left(a^2-1\right) R_n'(a)+\left(a^2-1\right) R_n(a){}^2+2 a (\alpha +n) R_n(a)\right)}{2 \left(\left(a^2-1\right) R_n(a)+a (2 \alpha +2 n+1)\right)},\label{MCr}\\
\sum_{j=0}^{n-1}R_j(a)&=\frac{2 \left(a^2-1\right) (\alpha +n) r_n(a)+\left(4 (\alpha +n)^2-1\right) \beta _n(a)-n (2 \alpha +n)}{a \left(a^2-1\right)}\label{MCSR}
\end{align}
and
\begin{equation}\label{MCb}
\beta_n(a)=\frac{1}{2n+2 \alpha -1}\left[\frac{\left(r_n(a)+n\right) \left(r_n(a)+2 \alpha +n\right)}{a R_n(a)+2 \alpha +2 n+1}-\frac{a r_n(a){}^2}{R_n(a)}\right].
\end{equation}
Hence, the coefficients of \eqref{MC1}  depend only on $R_n(a)$ and $R_n'(a)$.

\begin{proposition}\label{MCc}
For large $n$ and $a>0$ we have
\begin{align}\label{MCR}
R_n(a)=&\frac{2 a n}{1-a^2}+\frac{2 \alpha  a+a+1}{1-a^2}+\frac{4 \alpha ^2 a^2-a^2+1}{8 a^2 n^2}\notag\\
&-\frac{a (2 \alpha +1) \left[(2 \alpha -1) a^2 (a+2 \alpha +1)+1\right]+1}{8 a^3 n^3}+\mathcal{O}\left(\frac{1}{n^{4}}\right),\quad n\rightarrow \infty.
\end{align}
\end{proposition}
\begin{proof}
The second order differential equation for $R_n(a)$ can be obtained from the following system  (see \cite{MC2018}):
$$   R_n'(a)=R_n^2(a)+\frac{2a^2(n+\alpha)-2(\alpha^2-1)r_n(a)}{a(a^2-1)}R_n(a)-\frac{2(2n+2\alpha+1)}{a^2-1}r_n(a),
$$
\begin{align*}
r_n'(a)=&\frac{\left[(1-a^2)r_n^2(a)+2(n+\alpha)r_n(a)+n^2+2n\alpha\right]R_n(a)-a(2n+2\alpha+1)r^2_n(a)}
{(2n+2\alpha-1)R_n(a)(aR_n(a)+2n+2\alpha+1)}\\
&-\frac{2(n+\alpha)R_n(a)r_n(a)+n(n+2\alpha)R_n(a)}{(1-a^2)(aR_n(a)+2n+2\alpha+1)}.
\end{align*}
 If we consider the non-derivative part and terms with $n^4$, $n^3$ and $n^2$, then we have
\begin{align*}
      &4 n^2\widetilde{R}_{n}(a)^2 \left(a^2 \left(4 \left(6 a^2-5\right) \alpha ^2+8 \left(a^2-1\right) \alpha  (2 n+3)+\left(a^2-1\right) (4 n (n+2)+5)\right)\right.\\&\left.+\widetilde{R}_{n}(a) \left(2 \left(2 a^5-3 a^3+a\right) (6 \alpha +2 n+3)+\left(6 a^6-12 a^4+7 a^2-1\right) \widetilde{R}_{n}(a)\right)\right)=0.
\end{align*}
with the solution
\begin{align*}
 \widetilde{R}_{n1,2}(a)=& -\frac{2 a n}{a \left(2a\pm\sqrt{2-2 a^2}\right)-1}+\frac{\left(\sqrt{2}\pm6 a \sqrt{1-a^2} \left(1-2 a^2\right)\right) (2 \alpha +1)}{2 \sqrt{1-a^2} \left(6 a^4-6 a^2+1\right)}\\
 &~~~~~~~~~~~~~~~~~~\mp\frac{2 a^2 \left(a^2-4 \alpha ^2-1\right)+(2 \alpha +1)^2}{8 \sqrt{2} a^2 \left(1-a^2\right)^{3/2} n}+\mathcal{O}\left(\frac{1}{n^{2}}\right),\; n\rightarrow\infty.
\end{align*}

Assuming that  $R_n(a)$ is of the following form:
$$R_n(a)=\sum_{j=0}^\infty b_{j}(a)n^{1-j},\quad n\rightarrow\infty,$$
and substituting the  expression above into the second order differential equation for $R_n(a)$, we obtain  \eqref{MCR} as $n\rightarrow\infty$.
\end{proof}

\begin{theorem}
Let $n\rightarrow\infty$.  The differential equation for polynomials $P_n(x)$ orthogonal with respect to $(1-x^2)^\alpha(1-\chi(-a,a)(x))$ over $[-1,1]$ reduces to the geeneral Heun equation
\begin{equation}\label{MC2}
\widehat{P}_n''(u)+\left(-\frac{1}{2u}+\frac{\alpha +1}{u-1}+\frac{1}{u-t}\right)\widehat{P}_n'(u)-\frac{u (2 \alpha +n+1)n+n \sqrt t}{4u\left(u-1\right) \left(u-t\right)}\widehat{P}_n(u)=0,
\end{equation}
where $\widehat{P}_n(u):=P_n(\sqrt u)$.
\end{theorem}
\begin{proof}
Substituting \eqref{MCr}--\eqref{MCb} into \eqref{MC1} and sending $n\rightarrow \infty $, we plug in  the asymptotic expression for $R_n(a)$  \eqref{MCR} and find that the polynomials  $P_n(z)$ satisfy
\begin{equation}\label{MC3}
P_n''(z)+\left(\frac{2 z}{z^2-a^2}+\frac{2 (\alpha +1) z}{z^2-1}-\frac{2}{z}\right)P_n'(z)-\frac{n a+z^2 (2 \alpha +n+1)n}{\left(z^2-1\right) \left(z^2-a^2\right)}P_n(z)=0.
\end{equation}
Let
$$z=\sqrt{u},\quad a^2=t.$$
Then $\widehat{P}_n(u):=P_n(\sqrt u)$ satisfy the general Heun equation \eqref{MC2} with parameters
$$\widetilde{\gamma}=-1/2,\quad \widetilde{\delta}=\alpha+1,\quad \widetilde{\epsilon}=1,\quad \widetilde{\alpha}\widetilde{\beta}=-n(n+2\alpha+1)/4,\quad \widetilde{q}=\sqrt t n/4.$$
\end{proof}
\begin{corollary}
Equation \eqref{MC3}   reduces to the Jacobi differential equation when $a=0$:
\begin{equation}\label{MC4}
P_n''(z)+\frac{2 (\alpha +1) z}{z^2-1}P_n'(z)-\frac{(2 \alpha +n+1)n}{z^2-1 }P_n(z)=0.
\end{equation}
\begin{proof}
The weight $(1-x^2)^\alpha$ is the classical Jacobi weight in case $\beta=\alpha$. See the proof of  Corollary  \ref{Zc} for $x^\alpha(1-x)^\beta$ and Chen, Ismail \cite{ChenIsmail2005} for $(1-x)^\alpha(1+x)^\beta$.
\end{proof}
\end{corollary}

\begin{rem}
When $n$ is large and $a>0$ we have
\begin{align*}
r_n(a)=&-\frac{a n}{a+1}-\frac{a \alpha }{a+1}+\frac{4 \alpha ^2 a^2-a^2+1}{8 a n}+\mathcal{O}\left(\frac{1}{n^{2}}\right),\\
\beta _n(a)=&\frac{1}{4} (a-1)^2-\frac{(a-1)^2 (a+1) \left(a \left(4 \alpha ^2-1\right)-1\right)}{16 a n^2}+\mathcal{O}\left(\frac{1}{n^{3}}\right),\\
\sum_{j=0}^{n-1}R_j(a)=&\frac{a n^2}{1-a^2}+\frac{2 n a \alpha }{1-a^2}+\frac{1-a}{4 a(a+1)}+\frac{(a-1) \left(a^3 \left(4 \alpha ^2-1\right)-1\right)}{16 a^3 n^2}+\mathcal{O}\left(\frac{1}{n^3}\right).
\end{align*}
\end{rem}

\subsection{$x^\alpha{\rm e}^{-x}(A+B\theta(x-t)),\; x\in[0,\infty),\alpha,t>0,A\geq0,A+B\geq0$}
For the weight  $w(x)=(A+B\theta(x-t))x^\alpha{\rm e}^{-x}$ the second order differential equation for $P_n(z,t)$ reads
\begin{align}\label{LC1}
P_n''(z)+Q_n(z,t)P_n'(z) +S_n(z,t)P_n(z)=0,
\end{align}
where
\begin{align*}
Q_n(z,t)&=\frac{\alpha+1}{z}+\frac{1}{z-t}-\frac{1}{z-t+tR_n(t)}-1,\\
S_n(z,t)&=\frac{n}{z}-\frac{t \left[r_n(t)+n R_n(t)\right]}{z (z-t) \left(t R_n(t)-t+z\right)}+\frac{t }{z (z-t)}\sum_{j=0}^{n-1}R_j(t).
\end{align*}
The auxiliary quantities satisfy
 $\sum_{j=0}^{n-1}R_j(t)$ and $r_n(t)$ satisfy
\begin{align}
  r_n(t)&=\frac{tR'_n(t)-(2n+\alpha-t+tR_n(t))R_n(t)}{2},\label{LCr}\\
\sum_{j=0}^{n-1}R_j(t)&=\frac{\beta_n(t)-tr_n(t)-n(n+\alpha)}{t},\label{LCSR}
\end{align}
and
\begin{equation}\label{LCb}
\beta_n(t)=\frac{1}{1-R_n(t)}\left[(2n+\alpha)r_n(t)+n(n+\alpha)+\frac{r_n(t)^2}{R_n(t)}\right].
\end{equation}
Here
\begin{align*}
  &R_n(t):=B\frac{P_n(t,t)^2 t^\alpha{\rm e}^{-t}}{h_n(t)},  \\
 & r_n(t):=B\frac{P_n(t,t)P_{n-1}(t,t) t^\alpha{\rm e}^{-t}}{h_n(t)}
\end{align*}
and $R_n(t)$ satisfies
\begin{align}\label{BLCR2}
  R_n''=&\frac{1}{2}\left(\frac{1}{R_n-1}+\frac{1}{R_n}\right)R_n'{}^2-\frac{R_n'}{t}-\frac{\alpha^2}{2t^2}\frac{R_n}{R_n-1}\notag\\
  &+(2n+\alpha+1)\frac{R_n(R_n-1)}{t}+\frac{R_n(R_n-1)(2R_n-1)}{2}.
\end{align}
See Basor, Chen \cite{BC2009}.

In this paper we consider two cases, Lyu, Chen \cite{LC2} with $A=0$, $B=1$ and Lyu, Chen \cite{LC1} with $A=1, B=-1$ .

\vspace{0.3cm}

{ \bf The case  $A=0$, $B=1$.}
\vspace{0.3cm}

To study the large $n$ behavior of $R_n(t)$, we first recall some results from  \cite{LC2}.

\begin{proposition}\cite{LC2}
The function
$$R(s):=\lim_{n\rightarrow\infty}R_n\left(\frac{s}{4n}\right)$$
 satisfies the following second order differential equation:
$$R''(s)=\left(\frac{1}{R(s)-1}+\frac{1}{R(s)}\right)\frac{R'(s)^2}{2}-\frac{R'(s)}{s}+\frac{R(s)(R(s)-1)}{2s}-\frac{\alpha^2R(s)}{2s^2(R(s)-1)},$$
and it has the following expansion
\begin{align}\label{LCR}
  R(s)=1-\alpha s^{-\frac{1}{2}}&-\frac{\alpha}{8}s^{-\frac{3}{2}}-\frac{\alpha^2}{4}s^{-2}\notag\\
  &-\left(\frac{3\alpha^3}{8}
  +\frac{27\alpha}{128}\right)s^{-\frac{5}{2}}+\mathcal{O}\left(s^{-3}\right),\quad s\rightarrow \infty.
\end{align}
\end{proposition}


\begin{theorem}\label{LCT}
Let $n\rightarrow\infty$ and $s=4nt$ be fixed. Then  for large $s$, the polynomials $P_n(x)$ orthogonal with respect to $\theta(x-t)x^\alpha{\rm e}^{-x}$ over $[0,\infty)$ satisfy the double confluent Heun equation
\begin{align}\label{LC3}
P_n''(z)+\Big(\frac{s-\alpha  \sqrt{s}}{4 n z^2}+\frac{\alpha +1}{z}-1\Big)P_n'(z)+\frac{4 n z+(\sqrt{s}-\alpha)^2}{4 z^2}P_n(z)=0.
\end{align}
\end{theorem}

\begin{proof}
Substituting \eqref{LCr}--\eqref{LCb} into \eqref{LC1}, sending $n\rightarrow\infty$ and combining with \eqref{LCR}, for large $s$, we have
\begin{align*}
 \mathbb{ Q}_n(z,s) &=\frac{s-\alpha  \sqrt{s}}{4 n z^2}+\frac{\alpha +1}{z}-1 +\mathcal{O}\left(s^{-1/2}\right),\quad s\rightarrow \infty,\\
  \mathbb{S}_n(z,s) &=\frac{4 n z+(\sqrt{s}-\alpha)^2}{4 z^2}+\mathcal{O}\left(s^{-1/2}\right),\quad s\rightarrow \infty.
\end{align*}
Then \eqref{LC1} is a double confluent Heun equation with parameters
$$\widetilde{\gamma}=\frac{s-\alpha\sqrt s}{4n},\quad \widetilde{\delta}=\alpha+1,\quad \widetilde{\epsilon}=-1,\quad \widetilde{a}=n,\quad \widetilde{q}=-\frac{(\sqrt{s}-\alpha)^2}{4}.$$

\end{proof}

\begin{rem}For large $s$ we have
\begin{align}\label{}
  r_n(s/(4n))&=-n-\frac{\alpha }{2}+\frac{2\alpha ^2+4n \alpha+\alpha}{4\sqrt s}+\frac{\alpha   (2 \alpha +4 n+3)}{32s^{3/2}}\notag\\&~~~~~~~~~~~~~~~~~~~~~~~~~~~~~~+\frac{\alpha ^2 (\alpha +2 n+2)}{8 s^2}+\mathcal{O}\left(s^{-5/2}\right) ,\label{LCr1}\\
 \beta_n(s/(4n))&=n(n+\alpha )+\frac{\alpha ^2}{4}-\frac{\alpha \sqrt s }{4}+\frac{3\alpha}{32\sqrt s} +\frac{\alpha ^2}{8 s}   \notag\\&~~~~~~~~~~~~~~~~~~~~~~~~~~~~~~~~+\frac{5\alpha(16 \alpha ^2+9) }{512s^{3/2}}+\mathcal{O}\left(s^{-3/2}\right),\label{LCb1}\\
 \sum_{j=0}^{n-1}\mathbb{R}_j(s/(4n))&=n+\frac{\alpha}{2}- \frac{\alpha(4 n+\alpha)}{2\sqrt s}+\frac{\alpha ^2 n}{s}+\frac{\alpha(4n- \alpha)}{16s^{3/2}}\notag\\
 &~~~~~~~~~~~~~~~~~~~~~~~~~~~~~~~~~~~+\frac{\alpha^2(2n-\alpha)}{8s^2}+\mathcal{O}\left(s^{-5/2}\right).\label{LCSR1}
\end{align}
\end{rem}

\vspace{0.3cm}

{ \bf The case   $A=1$, $B=-1$.}
\vspace{0.3cm}

\begin{proposition}
As $n\rightarrow\infty$, the quantity $R_n(t)$ has the following asymptotic expression:
\begin{equation}\label{1LCR}
 R_n(t)=-\frac{2 n}{t}+\frac{t-2 (\alpha +1)}{2 t}+\frac{\alpha ^2}{8 n^2}+\frac{-4 \alpha ^3-2 \alpha ^2 (t+2)+t}{32 n^3}+\mathcal{O}\left(n^{-4}\right).
\end{equation}

\begin{proof}
Neglecting the derivative terms in \eqref{BLCR2} and replacing $R_n(t)$ by $ \widetilde{R}_n(t)$, we obtain
\begin{align*}
 2 t^2 \widetilde{R}_n(t)^3+t (2 \alpha +4 n-5 t+2) \widetilde{R}_n(t)^2&-4 t (2n+\alpha +1-t)\widetilde{R}_n(t)\\&-\alpha ^2+2 t (\alpha +2 n+1)-t^2=0.
\end{align*}
The solution to the equation above when $n\rightarrow \infty$ is
\begin{align*}
\widetilde{R}_{n1}(t)&=-\frac{2 n}{t}+\frac{t-2 (\alpha +1)}{2 t}+\frac{\alpha ^2}{8 n^2}-\frac{\alpha ^2 (2 \alpha +t+2)}{16 n^3}+\mathcal{O}\left(n^{-4}\right),\\
\widetilde{R}_{n2}(t)&=1\pm\frac{\alpha  }{2 \sqrt{nt}}\mp\frac{ \alpha  (2 \alpha +t+2)}{16 \sqrt{t}n^{3/2}}-\frac{\alpha ^2}{16n^2}+\mathcal{O}\left(n^{-{5/2}}\right).
\end{align*}
Since  $R_n(t)=-P_n(t,t)^2 t^\alpha{\rm e}^{-t}/h_n(t)<0 $ we assume that $R_n(t)$ has the following expression:
$$R_n(t)=\sum _{j=0}^{\infty}a_jn^{1-j},\quad n\rightarrow\infty.$$
Substituting the expression above into \eqref{BLCR2} we obtain \eqref{1LCR}.
\end{proof}

\end{proposition}

\begin{theorem}
Sending $n$ to infinity, the polynomials $\widetilde{P}_n(u):=P_n(tu)$  satisfy the confluent Heun equation
\begin{align}\label{2LC3}
\widetilde{P}_n''(u)+\Big(\frac{\alpha +1}{u}+\frac{1}{u-1}-t\Big)\widetilde{P}_n'(u)+\frac{ntu- n( n+\alpha+1+t/2 )}{u (u-1) }\widetilde{P}_n(u)=0.
\end{align}
Here $P_n(x)$ are orthogonal with respect to $(1-\theta(x-t))x^\alpha{\rm e}^{-x}$ over $[0,\infty)$.
\end{theorem}
\begin{proof}
Substituting \eqref{LCr}--\eqref{LCb} into \eqref{LC1}, sending $n\rightarrow\infty$, combining with \eqref{1LCR}  and setting  $\widetilde{P}_n(u)=P_n(tu)$ we obtain the confluent Heun equation with parameters
$$\widetilde{\gamma}=\alpha+1,\quad \widetilde{\delta}=1,\quad \widetilde{\epsilon}=-t,\quad \widetilde{a}=nt,\quad \widetilde{q}=n( n+\alpha+1+t/2 ).$$
\end{proof}

\begin{rem}For large $n$ we have
\begin{align*}
  r_n(t)&=-\frac{n}{2}+\frac{t-2 \alpha}{8}+\frac{\alpha ^2}{8 n}+\frac{t-2 \alpha ^2(\alpha+t)}{32 n^2}+\mathcal{O}\left(n^{-3}\right) ,\\
 \beta_n(t)&=\frac{t^2}{16}+\frac{t^2(1-2 \alpha ^2)}{64 n^2}+\mathcal{O}\left(n^{-3}\right),\\
 \sum_{j=0}^{n-1}\mathbb{R}_j(t)&=-\frac{n^2}{t}+n \left(\frac{1}{2}-\frac{\alpha }{t}\right)+\frac{\alpha }{4}-\frac{t}{16}-\frac{\alpha ^2}{8 n}+\frac{4 \alpha ^3+2 \alpha ^2 t-t}{64 n^2}+\mathcal{O}\left(n^{-3}\right).
\end{align*}
\end{rem}

\section{Conclusion}
In this paper we considered the eight kinds of weight functions for monic orthogonal  polynomials $P_n(x)$. These polynomials satisfy linear second order differential equations and we showed that they reduce to Heun equations as $n\to\infty$. In this way we obtained six confluent Heun equations, three double confluent Heun equations and a general Heun equation.

\begin{rem}
For the deformed Freud  weight we will obtain the biconfluent Heun equation
\begin{equation}
\frac{d^2u}{dz^2}+\left(\frac{\gamma}{z}+\delta+\epsilon z\right)\frac{du}{dz}+\left(\frac{\alpha z-q}{z}\right)u=0.
\end{equation}
See the work of Clarkson and Jordaan \cite{Cl2018} for the deformed Freud weight $$|x|^{2\lambda+1}{\rm e}^{-x^4+tx^2},\; \lambda>-1,\;x\in\mathbb{R}.$$ The  biconfluent Heun equation was obtained in \cite[p. 165]{Cl2018} with parameters
$$ \gamma=1+\lambda,\quad \delta=\frac{\sqrt 2 t}{2},\quad \epsilon=-1,\quad \alpha=0,\quad q=-\frac{\sqrt6 n^{3/2}}{9}.$$
Also see Zhu and Chen \cite{ZC2019} for $$|x|^\alpha{\rm e}^{-N[x^2+s(x^4-x^2)]},\quad x\in\mathbb{R},$$ where the  biconfluent Heun equation \cite[Eq 6.22]{ZC2019} was obtained    with parameters
$$ \gamma=-\frac{\alpha+1}{2},\quad \delta=\frac{\sqrt 2 N(1-s)}{2},\quad \epsilon=1,\quad \alpha=0,\quad q=-\frac{\sqrt6 k^{3/2}}{9}.$$

At present we do not know the examples of weights that would lead to the  triconfluent Heun equation.
\end{rem}

\section{ Appendix}
\subsection{Integral identities}

Using the Coulomb fluid method requires numerous integral formulas. Here we list some integrals used in main text, which  can be found in \cite{CM2012}. For case of  $0<a<b$  we have
\begin{align}
  \int_a^b\frac{{\rm d}x}{\sqrt{(b-x)(x-a)}} &=\pi, \label{int1}\\
 \int_a^b\frac{x {\rm d}x}{\sqrt{(b-x)(x-a)}} &=\frac{a+b}{2}\pi, \label{int2}\\
 \int_a^b\frac{{\rm d}x}{x\sqrt{(b-x)(x-a)}} &=\frac{\pi}{\sqrt{ab}}, \label{int3}\\
 \int_a^b\frac{{\rm d}x}{x^2\sqrt{(b-x)(x-a)}}x &=\frac{a+b}{2(ab)^{3/2}}\pi,\label{int4}\\
 \int_a^b\frac{{\rm d}x}{(x+t)\sqrt{(b-x)(x-a)}} &=\frac{\pi}{\sqrt{(b+t)(t+a)}}.\label{int5}
\end{align}

\section{Acknowledgements}

\noindent
 L. Zhan, Y. Chen would like to thank the Science and Technology Development Fund of the Macau SAR for generous support in providing FDCT 130/2014/A3 and FDCT 023/2017/A1. We would also like to thank the University of Macau for generous support via MYRG 2014-00011 FST, MYRG 2014-00004 FST and MYRG 2018-00125 FST.

 G. Filipuk  acknowledges the support of Alexander von Humboldt Foundation.   The support of National Science Center (Poland) via NCN OPUS grant $2017/25/B/BST1/00931$ is acknowledged.

\end{document}